\documentclass[onefignum,onetabnum,a4paper]{siamart171218}

\usepackage{ amssymb }

\newcommand{\pr}[1]{\left(#1\right)}
\newcommand{\br}[1]{\left[#1\right]}
\newcommand{\cb}[1]{\left\{#1\right \}}
\newcommand{\nm}[1]{\left\|#1\right \|}

\newcommand{\abs}[1]{\left|#1\right|}

\newcommand{\uint}{\int_0^1}

\newcommand{\hgn}{\hat{g}_n}
\newcommand{\hfn}{\hat{f}_n}

\newcommand{\E}{\mathbb{E}}
\newcommand{\R}{\mathbb{R}}
\renewcommand{\P}{\mathbb{P}}

\newcommand{\de}{\delta}

\usepackage{lipsum}
\usepackage{amsfonts}
\usepackage{graphicx}
\usepackage{epstopdf}
\usepackage{algorithmic}
\usepackage{bbm}
\ifpdf
  \DeclareGraphicsExtensions{.eps,.pdf,.png,.jpg}
\else
  \DeclareGraphicsExtensions{.eps}
\fi

\newsiamremark{remark}{Remark}
\newsiamremark{hypothesis}{Hypothesis}
\crefname{hypothesis}{Hypothesis}{Hypotheses}
\newsiamthm{claim}{Claim}

\title{Is speckle noise more challenging to mitigate than additive noise?}
\author{ Reihaneh Malekian\thanks{Columbia University, New York, NY, USA
 (\email{rm3942@columbia.edu}).}
 \and
 Hao Xing\thanks{CUNY Graduate Center, New York, NY, USA
 (\email{hao.xing43@gc.cuny.edu}).}
 \and
 Arian Maleki\thanks{Columbia University, New York, NY, USA
  (\email{mm4338@columbia.edu}).}
}
\usepackage{amsopn}

\makeatletter
\newcommand*{\addFileDependency}[1]{% argument=file name and extension
  \typeout{(#1)}% latexmk will find this if $recorder=0 (however, in that case, it will ignore #1 if it is a .aux or .pdf file etc and it exists! if it doesn't exist, it will appear in the list of dependents regardless)
  \@addtofilelist{#1}% if you want it to appear in \listfiles, not really necessary and latexmk doesn't use this
  \IfFileExists{#1}{}{\typeout{No file #1.}}% latexmk will find this message if #1 doesn't exist (yet)
}
\makeatother

\newcommand*{\myexternaldocument}[1]{%
    \externaldocument{#1}%
    \addFileDependency{#1.tex}%
    \addFileDependency{#1.aux}%
}
\ifpdf
\hypersetup{
  pdftitle={Is speckle noise more challenging to mitigate than additive noise?},
  pdfauthor={Reihaneh}
}
\fi

\myexternaldocument{ex_supplement}

\usepackage{natbib}
\usepackage{pifont}
\usepackage{amsmath}
\usepackage{graphicx,upgreek}
\usepackage{caption}
\begin{document}
\newtheorem{defn}{Definition}
\maketitle

% REQUIRED

\begin{abstract}
 % This is an example SIAM \LaTeX\ article.
We study the problem of estimating a function in the presence of both speckle and additive noises, commonly referred to as the de-speckling problem. Although additive noise has been thoroughly explored in nonparametric estimation, speckle noise, prevalent in applications such as synthetic aperture radar, ultrasound imaging, and digital holography, has not received as much attention. Consequently, there is a lack of theoretical investigations into the fundamental limits of mitigating the speckle noise.This paper is the first step in filling this gap. 

Our focus is on investigating the minimax estimation error for estimating a $\beta$-H\"older continuous function and determining the rate of the minimax risk. Specifically, if $n$ represents the number of data points, $f$ denotes the underlying function to be estimated, $\hat{\nu}_n$ is an estimate of $f$, and $\sigma_n$ is the standard deviation of the additive Gaussian noise, then $\inf_{\hat{\nu}_n} \sup_f \mathbb{E}_f\| \hat{\nu}_n - f \|^2_2$ decays at the rate $\pr{\max(1,\sigma_n^4)/n}^{\frac{2\beta}{2\beta+1}}$. Comparing this rate with the rate achieved under purely additive noise, namely $\pr{\sigma_n^2/n}^{\frac{2\beta}{2\beta+1}}$, leads to the following insights: (i) When \( \sigma_n = \omega(1) \), the additive noise appears to be the dominant component in the de-speckling problem. However, the presence of speckle noise significantly complicates the task of mitigating its effects. As a result, the risk increases from the rate \( \left( \sigma_n^2 / n \right)^{\frac{2\beta}{2\beta+1}} \), which characterizes the problem with only additive noise, to \( \left( \sigma_n^4 / n \right)^{\frac{2\beta}{2\beta+1}} \) in the presence of both speckle and additive noise. (ii) When \( \sigma_n = o(1) \), the variance of the additive noise does not contribute to the risk in the de-speckling problem. This suggests that, in this regime, speckle noise is the primary bottleneck. Interestingly, the resulting risk rate matches the rate for mitigating purely additive noise with \( \sigma_n = \Theta(1) \). (iii) When $\sigma_n = \Theta(1)$, the two rates coincide, suggesting that both the speckle noise and additive noise are contributing to the overall error.

\end{abstract}

% REQUIRED

% REQUIRED
%\begin{AMS}
 % 68Q25, 68R10, 68U05
%\end{AMS}

\section{Introduction}

\subsection{Motivation and Problem Set-up}
 Denoising—the task of removing additive noise superimposed on a function—is a fundamental problem in signal processing and statistics. It is mathematically formulated as the problem of estimating an unknown function $f$ from noisy observations of the form  
\begin{equation}\label{eq:denoise:model}
\widetilde{y}_i = f(x_i) + w_i, \quad i = 1, 2, \ldots, n,
\end{equation}
where the noise terms $w_i$ are independent and identically distributed as  $\mathcal{N}(0, \sigma_n^2)$. Due to its significance in a wide range of applications, this problem has been the subject of extensive theoretical study, with substantial contributions dating back at least seven decades \cite{Na64, Wa64, Stone77, Stone80, Stone82, Ka79, Cleveland79, Fano1952, lecam1953, Ne85, NPT85, KP92, donoho1994ideal, Mainbook, donoho1999wedgelets, maleki2012suboptimality, maleki2013anisotropic, arias2012oracle, devore2025optimalrecoverymeetsminimax}. 
For a comprehensive summary of historical developments, we refer the reader to \cite{Nonp} and \cite{devore2025optimalrecoverymeetsminimax} and the references therein.

 Despite a rich theoretical literature on the denoising problem, the related challenge of de-speckling has remained virtually unexamined from a theoretical standpoint. Speckle noise is among the most significant sources of distortion in coherent imaging systems, including Synthetic Aperture Radar (SAR) \citep{Lopez-MartinezFabregas}, Optical Coherence Tomography (OCT) \citep{SXK}, ultrasound imaging \citep{ABT}, and digital holography \citep{dh}. Although numerous practical de-speckling methods have been developed, the theoretical foundations and fundamental limits of this problem remain largely unexplored. In this work, we aim to present the first theoretical analysis of the de-speckling problem.

Towards this goal, we consider the problem of recovering a function $f$ from its noise-corrupted samples in the presence of both additive and multiplicative noise. Specifically, suppose we are provided with observations, denoted as 
\begin{equation}\label{speckle noise model}
y_i = f\left(x_i\right) \xi_i + \tau_i, \ \ \ \ \ i = 1, 2, \ldots, n.
\end{equation}

Here, $\xi_i \overset{iid}{\sim} \mathcal{N}(0, 1)$ represents the multiplicative noise and $\tau_i \overset{iid}{\sim} \mathcal{N}(0, \sigma_n^2)$ denotes the additive noise. The objective is to estimate the function $f$ based on these observations.  

\begin{remark}
 In our model, we assume that the speckle noise follows a Gaussian distribution. This is motivated by the fact that speckle arises from the interference of coherent wavefronts scattered by microscopic surface variations. By the central limit theorem, physicists have shown that in most cases, speckle noise can be well-approximated by zero-mean Gaussian noise \cite{goodman1975statistical}. In practice, this model is  referred to as fully-developed speckle \cite{argenti2013tutorial}.
\end{remark}

\begin{remark}\label{remark: scaling}
In \cref{speckle noise model}, we have fixed the standard deviation of the speckle noise to be $1$. This is due to the fact that increasing the standard deviation of the multiplicative noise is the same as decreasing the standard deviation of the additive noise. More specifically, if $\xi_{i} \overset{iid}{\sim} \mathcal{N}(0, \sigma_\xi^2)$, then we can rescale our data to obtain
\[
\check{y}_i = \frac{y_i}{\sigma_\xi} =  f\left(x_i \right)\check{\xi}_{i} + \frac{\tau_i}{\sigma_\xi}, 
\]
where $\check{\xi}_i \overset{iid}{\sim} \mathcal{N}(0,1)$, and $\frac{\tau_i}{\sigma_\xi} \overset{iid}{\sim} \mathcal{N}(0, \frac{\sigma_n^2}{\sigma_\xi^2} )$.
In other words, assuming that $\sigma_\xi$ is known, the problem of estimating $f$ from $Y_1, \ldots, Y_n$ is equivalent to the problem of estimating $f$ from $\check{Y}_1, \ldots, \check{Y}_n$ in which the variance of the multiplicative noise is equal to $1$. Hence, without loss of generality we set $\sigma_\xi^2 =1$. 
\end{remark}

\vspace{3mm}
In this paper, we aim to answer the following questions:
\begin{itemize}
\item How well can we estimate $f$ from the observations $y_1, y_2, \ldots, y_n$? 
\item How is the complexity of this problem compared with the complexity of the denoising problem in \eqref{eq:denoise:model}. 
\end{itemize}

To address these questions, following the long history of the de-noising problem (see \cite{Nonp}), we adopt the minimax framework. The first step in this framework is to consider a class of functions for $f(\cdot)$. In \cite{Nonp}, it is assumed that the function $f$ is $\beta$-H{\"o}lder continuous, i.e. it belongs to the following class of functions:
\[
\Sigma(\beta, L) := \{ f: [0,1] \rightarrow \R \ : \ |f^{(k)}(x)- f^{(k)}(y)| \leq L |x-y|^{\beta-k}    \},
\]
where $L$ and $\beta$ are two positive constants, and $k$ is the largest integer strictly less than $\beta$. However, due to some specific features of the speckle noise problem, throughout this paper we consider the functions $f$ to be in the following subclass of $\beta$-H{\"o}lder functions: 
\begin{equation}
    \Sigma_{\mathfrak{h}}(\beta, L) := \{ f\in \Sigma(\beta, L) : 0<\mathfrak{h} \leq f(x) \leq 1 \}. 
\end{equation}

In \cref{app:model_justification}, we present the rationale for the choices we have made in the definition of  $ \Sigma_{\mathfrak{h}}(\beta, L)$. 

Using the minimax framework we measure the difficulty level of our estimation problem as:
\begin{equation}\label{eq:multiplicative:minimax:1}
R_2(\Sigma_{\mathfrak{h}}(\beta, L), \sigma_n) :=\inf_{\hat{\nu}} \sup_{f \in \Sigma_{\mathfrak{h}}(\beta, L)} \mathbb{E}_f \|f- \hat{\nu}\|^2_2,
\end{equation}
and 
\begin{equation}\label{eq:multiplicative:minimax:2}
R_{\infty}(\Sigma_{\mathfrak{h}}(\beta, L), \sigma_n) := \inf_{\hat{\nu}} \sup_{f \in \Sigma_{\mathfrak{h}}(\beta, L) } \mathbb{E}_f \|f- \hat{\nu}\|_{\infty},
\end{equation}
where in both equations $\hat{\nu}$ denotes an estimator of function $f$ based on the observations $y_1, y_2, \ldots, y_n$ in \cref{speckle noise model}. 

Therefore, the aim of the rest of this paper is to calculate $R_{2}(\Sigma_{\mathfrak{h}}(\beta, L), \sigma_n)$ and $R_{\infty}(\Sigma_{\mathfrak{h}}(\beta, L), \sigma_n)$ defined in \cref{eq:multiplicative:minimax:1} and \cref{eq:multiplicative:minimax:2} and compare them with the corresponding results for the denoising problem.

To distinguish the classical model (with only additive noises) from ours, we shall use $\widetilde{R}_2(\Sigma_{\mathfrak{h}}(\beta, L), \sigma_n)$ and  $\widetilde{R}_{\infty}(\Sigma_{\mathfrak{h}}(\beta, L), \sigma_n)$ for the denoising problem where estimators are based on $\tilde{y}_1, \tilde{y}_2, \ldots, \tilde{y}_n$. As mentioned before, the minimax risk has been used extensively in the past for the denoising problem. For instance, \cite{Nonp} (Corollary 1.2, Theorem 1.8), and \cite{Mainbook}(Example 2.7.1, Example 2.7.3) lead to the following conclusion:

\begin{theorem}\label{mimimax error for the classical additive noise model} For the denoising problem mentioned in \eqref{eq:denoise:model}, we have 
\begin{align}
\begin{split}  \widetilde{R}_2(\Sigma_{\mathfrak{h}}(\beta, L), \sigma_n) :=& \inf_{\hat{\nu}} \sup_{f \in \Sigma_{\mathfrak{h}}(\beta, L)} \mathbb{E}_f \|f- \hat{\nu}\|_2^2 \\
=&
\Theta_{\beta, L}\pr{ \min\cb{1, \max\cb{n^{-2\beta},\pr{\frac{\sigma_n^2}{n}}^{\frac{2\beta}{2\beta+1}}}}}.
\end{split}
\end{align}    
Furthermore, if $\sigma_n^2=O(n^{1-\eta})$ for some $\eta\in (0,1)$, then
\begin{align}\label{eq:addtive:minimax:2}
\begin{split}
    \widetilde{R}_{\infty}(\Sigma_{\mathfrak{h}}(\beta, L), \sigma_n) 
:= & \inf_{\hat{\nu}} \sup_{f \in \Sigma_{\mathfrak{h}}(\beta, L) } \mathbb{E}_f \|f- \hat{\nu}\|_{\infty} \\
=&
     \Theta_{\beta, L}\pr{ \max\cb{n^{-2\beta},\pr{\frac{ \sigma_n^2 \log n}{n}}^{\frac{\beta}{2\beta+1}}}}
\end{split}
\end{align}
where in these two equations $\hat{\nu}$ is an estimator that uses the data points from \cref{eq:denoise:model}.
\end{theorem}

Since the classical literature focuses only on the case where $\sigma_n= \Theta(1)$, we include a self-contained proof of \cref{mimimax error for the classical additive noise model} in \cref{proof of mimimax error for the classical additive noise model}.
 However, we should emphasize that our proof is a minor modification of the proof presented in Sections 1.6, 2.5 and 2.6 of \cite{Nonp}.  

The aim of the rest of this paper is to calculate $R_{\infty}(\Sigma_{\mathfrak{h}}(\beta, L), \sigma_n)$ and $R_{2}(\Sigma_{\mathfrak{h}}(\beta, L), \sigma_n)$ defined in \cref{eq:multiplicative:minimax:1} and \cref{eq:multiplicative:minimax:2} and compare them with the corresponding results for the denoising problem.

\subsection{Notations}
%\textcolor{red}{Will be revised later}
For $\alpha \in \mathbb{R}$, $\lfloor\alpha \rfloor$ (resp. $\lceil\alpha \rceil$) denotes the largest (resp. smallest) integer not greater (resp. bot less) than $\alpha$. 

For two sequences of numbers $\{a_n\}$ and $\{b_n\}$, indexed with $n$, we say $a_n=O(b_n)$, or equivalently $b_n=\Omega(a_n)$, if there exist constants $C>0$, and $M>0$, such that for all $n>M$, $|a_n|\le C |b_n|$. Also, $a_n=o(b_n)$, or equivalently $b_n=\omega(a_n)$ if $\lim_{n\to\infty} {a_n}/{b_n} =0$. Finally, we write $a_n = \Theta (b_n)$ if $a_n = O(b_n)$ and $b_n = O(a_n)$. 

The \emph{Hamming distance} between two vectors $u,v \in \{0,1\}^m$, is represented with $\rho (u,v)$, and is defined as  
\[
\rho(u,v) := \sum_{i=1}^m \mathbbm{1}\{u_i \neq v_i\},
\]
where $\mathbbm{1}\{\cdot\}$ denotes the indicator function. For a matrix $A$, $\lVert A \rVert_F$ and $\lVert A \rVert$ denote its Frobenius and operator norms, respectively.
For function $\phi: \mathbb{R} \rightarrow \mathbb{R}$, we use the following notations:
$$\lVert \phi \rVert^2_2 :=\int^{\infty}_{-\infty} \phi(x)^2 dx, $$
$$\phi^{*} :=\lVert \phi \rVert_{\infty}:= \sup_{x} |\phi(x)|.$$

For any function $f:\mathbb{R}\to \mathbb{R}$, we define 
$$\operatorname{supp}(f):= \{x \in \mathbb{R} : f(x)>0\}.$$

Suppose that $\nu'$ and $\nu''$ denote two different measures defined on random vector $y_1, \ldots, y_n$. Then we use the following notation for the likelihood ratio:
\begin{align}\label{lr:def}
\Lambda\left(\nu^{\prime}, \nu^{\prime \prime}\right):= \Lambda\left(\nu^{\prime}, \nu^{\prime \prime};y_1, \ldots, y_n\right)= \mathbb{P}_{\nu^{\prime}}(y_1, \ldots, y_n)/\mathbb{P}_{\nu^{\prime \prime}}(y_1, \ldots, y_n),
\end{align}
where $\mathbb{P}_{\nu^{\prime}}(y_1, \ldots, y_n)$ denotes the probability density function of $y_1, y_2, \cdots, y_n$ under the measure $\nu'$. 

\subsection{Organization of our paper}

The paper is organized as follows. We state our main theoretical results regarding the de-speckling problem in 
\cref{sec:main}. Furthermore, we compare the results of our minimax analysis with the results of \cref{mimimax error for the classical additive noise model}  ad clarify its implications.  The proofs of our main results will be presented in  \cref{proofs}.  This section begins with the statement of preliminary results required for our analysis, ensuring the self-containment of the section.
  Finally, we conclude in \cref{Conclusion}.
%and the conclusions follow in \cref{sec:conclusions}.

\section{Main contributions}\label{sec:main}

As described before, our main goal is to calculate $R_{\infty}(\Sigma_{\mathfrak{h}}(\beta, L), \sigma_n)$ and $R_{2}(\Sigma_{\mathfrak{h}}(\beta, L), \sigma_n)$ defined in \cref{eq:multiplicative:minimax:1} and \cref{eq:multiplicative:minimax:2} and compare them with the corresponding results for the denoising problem. Since in imaging applications the samples are equispaced, we consider the following model for the data:
\[
Y_{i}=f\left(i/n\right)\xi_{i} + \tau_i,
\]
where $\xi_i \overset{iid}{\sim} \mathcal{N}(0,1)$ and $\tau_i \overset{iid}{\sim} \mathcal{N}(0,{\sigma}^2_n)$. Our first result establishes the minimax risk $R_2 (\Sigma_{\mathfrak{h}}(\beta, L), \sigma_n)$.

\begin{theorem}\label{l_2}
We have
\begin{equation}\label{current bounds}
   R_2 \left(\Sigma_{\mathfrak{h}}(\beta, L), \sigma_n\right) =\Theta_{\beta, L, \mathfrak{h}}\pr{ \min\cb{1,\pr{\frac{\max(1,\sigma_n^4)}{n}}^{\frac{2\beta}{2\beta+1}}}}.
\end{equation}
\end{theorem}

\vspace{5mm}

    To compare \cref{l_2} with the minimax error analysis of the classical additive noise model (cf. \cref{mimimax error for the classical additive noise model}), we discuss the following three settings:
 \begin{enumerate}
\item Strong additive noise, $\sigma_n = \omega(1)$: In this case, we can see that $ R_2 \left(\Sigma_{\mathfrak{h}}(\beta, L), \sigma_n\right) = \Theta\left(\left(\frac{\sigma_n^4}{n}\right)^{\frac{2\beta}{2 \beta+1}}\right)$, while $\tilde{R}\left(\Sigma_{\mathfrak{h}}(\beta, L), \sigma_n\right) = \Theta\left(\left(\frac{\sigma_n^2}{n}\right)^{\frac{2\beta}{2 \beta+1}}\right)$. 
Comparing the two cases reveals that the error decays more slowly in the de-speckling problem. This suggests that, although additive noise may appear to be the dominant component, the presence of even seemingly small speckle noise significantly complicates the task of mitigating the additive noise. 

\item Medium additive noise, $\sigma_n = \Theta(1)$: In this case, we can see that $$ R_2 \left(\Sigma_{\mathfrak{h}}(\beta, L), \sigma_n\right) = \Theta (\tilde{R_2} \left(\Sigma_{\mathfrak{h}}(\beta, L), \sigma_n\right)).$$ In other words, order-wise, in this case, the de-speckling and de-noising problems are equally hard. 
 
\item Weak additive noise, $\sigma_n = o(1)$: In this case, we can see that 
$$ R_2 \left(\Sigma_{\mathfrak{h}}(\beta, L), \sigma_n\right) \ll \tilde{R_2} \left(\Sigma_{\mathfrak{h}}(\beta, L), \sigma_n\right).$$ Furthermore, the variance of the additive noise does not appear in $R_2 \left(\Sigma_{\mathfrak{h}}(\beta, L), \sigma_n\right)$. This indicates that, in this case, the bottleneck in estimating the function \( f(\cdot) \) from \( y_1, y_2, \ldots, y_n \) is the speckle noise, and that the contribution of the additive noise is negligible.

 \end{enumerate}

\iffalse
\begin{remark}
Note that the rates we have obtained for $R_2 \left(\Sigma_{\mathfrak{h}}(\beta, L), \sigma_n\right)$ do not depend on $\sigma_n$. In other words, even if $\sigma_n$ converges to zero very fast, it does not improve the decay rate, and the bottleneck is the multiplicative noise. 
\end{remark}

\begin{remark}
Ignoring the minor mismatch between the upper bound and the lower bound, roughly speaking, this theorem states that $R_2 \left(\Sigma_{\mathfrak{h}}(\beta, L), \sigma_n\right) \sim n^{-\frac{2\beta}{{2 \beta + 1}}}$. While our theoretical results obtain the rates for large values of $n$, the simulation results we obtain in \cref{simulations} confirm this observation even for not-so-large values of $n$. 
\end{remark}

\begin{remark}
Comparing the approximate decay rate $R_2 \left(\Sigma_{\mathfrak{h}}(\beta, L), \sigma_n\right) \sim n^{-\frac{2\beta}{{2 \beta + 1}}}$  with that of the additive noise, i.e. \cref{eq:addtive:minimax:1}, we realize that the rates are the same when $\sigma_n$ is a constant that does not decay with $n$. In other words, as far as the minimax rates of estimation are concerned, if the variances of the additive noise and multiplicative noise are at the same order, then the complexity of de-speckling and denoising problems are the same. 
\end{remark}
\fi

Our second theorem aims to calculate $R_{\infty} \left(\Sigma_{\mathfrak{h}}(\beta, L), \sigma_n\right)$.  

\begin{theorem}\label{l_infty}
We have for any $\eta\in (0,1)$ and $\sigma_n^4=o(n^{1-\eta})$,
\begin{align}
    R_{\infty} \left(\Sigma_{\mathfrak{h}}(\beta, L), \sigma_n\right) =
\Theta_{\beta,L,\eta, \mathfrak{h}}\pr{\pr{\max(1,\sigma_n^4)\pr{\log n \over n}}^{\beta \over 2\beta+1}} 
\end{align}
\end{theorem}

All the remarks we made about \cref{l_2} are also true about \cref{l_infty}. Hence, we do not repeat them here. 

%\begin{remark}
%Note that 
%Similarly, we can prove the same rates, when $\sigma_n$ is assumed to be zero. In other words, when the observations are assumed to be generated from the model
%\begin{equation}
%y_i = f\left(x_i\right) \xi_i , \ \ \ \ \ i = 1, 2, \ldots, n,
%\end{equation}
%and $\xi_i \overset{iid}{\sim} \mathcal{N}(0, 1)$. We skip the proofs as they are very similar to our proofs presented here.
%\end{remark}

\section{Proofs}\label{proofs}

This section is devoted to the proofs of \cref{l_2} and \cref{l_infty}. Before we start the proofs, we state a few preliminary results that will be used in our proofs. 

\subsection{Preliminaries}
In this section, we review some classical results that will be used in our main proofs. 

The first lemma characterizes one of the minimax risks  studied in this paper.

\begin{lemma}\label{lem:risk:monotonicity}
All the functions $R_{\infty} \left(\Sigma_{\mathfrak{h}}(\beta, L), \sigma_n\right)$,  $R_2 \left(\Sigma_{\mathfrak{h}}(\beta, L), \sigma_n\right)$, $\tilde{R}_{\infty} \left(\Sigma_{\mathfrak{h}}(\beta, L), \sigma_n\right)$, and $\tilde{R}_{\infty} \left(\Sigma_{\mathfrak{h}}(\beta, L), \sigma_n\right)$ are nondecreasing functions of $\sigma_n$. 
\end{lemma}
\begin{proof}
 The monotonicity of the minimax risk with respect to the variance of additive Gaussian noise in denoising problems is a well-established result \cite{donoho1994ideal}. However, for the sake of completeness—and because this monotonicity has not been established in the specific setting of our paper—we provide a proof for $R_2(\Sigma_{\mathfrak{h}}(\beta, L), \sigma_n)$ below. Consider the two cases,
\begin{align}
\begin{split}
    Y_i = f(i/n)\zeta_i + \tau_i,  \\
\tilde{Y}_i = f(i/n)\tilde{\zeta}_i + \tilde{\tau}_i.
\end{split}
\end{align}
where $\zeta_i \sim N(0,1)$, $\tilde{\zeta}_i \sim N(0,1)$, $\tau_i \sim N(0, \sigma^2_n)$, and $\tilde{\tau}_i \sim N(0, \tilde{\sigma}_n^2)$. Furthermore, our assumption is that $\tilde{\sigma}_n> \sigma_n$. Our goal is to show that
\[
R_2(\Sigma_{\mathfrak{h}}(\beta, L), \sigma_n) \leq R_2(\Sigma_{\mathfrak{h}}(\beta, L), \tilde{\sigma}_n).
\]
Suppose that $\hat{f} (\tilde{Y}_1, \tilde{Y}_2, \ldots, \tilde{Y}_n)$ is an arbitrary estimator that is based on $\tilde{Y}_1, \tilde{Y}_2, \ldots, \tilde{Y}_n$. Based on this estimator we are going to construct an estimator that can be applied to $Y_1, \ldots, Y_n$. We first generate independent random variables $\omega_i \sim N(0, \sqrt{\tilde{\sigma}_n^2 - \sigma_n^2})$, and construct the estimate: $\hat{f} ({Y}_1 + \omega_1, {Y}_2+\omega_2, \ldots, {Y}_n+\omega_n)$. By our construction we have that
\[
\mathbb{E} \|\hat{f} ({Y}_1 + \omega_1, {Y}_2+\omega_2, \ldots, {Y}_n+\omega_n)- f\|_2^2 = \mathbb{E} \|  \hat{f} (\tilde{Y}_1, \tilde{Y}_2, \ldots, \tilde{Y}_n)- f\|_2^2. 
\]
Let $\mathbb{E}_Y$ denote the conditional expectation given the random variables $Y_1, Y_2, \ldots, Y_n$. Then, 
\begin{align}
\begin{split}
    & \mathbb{E} \|\hat{f} ({Y}_1 + \omega_1, {Y}_2+\omega_2, \ldots, {Y}_n+\omega_n)- f\|_2^2  \\
= & \mathbb{E} \br{ \E_{Y}  \|\hat{f} ({Y}_1 + \omega_1, {Y}_2+\omega_2, \ldots, {Y}_n+\omega_n)- f\|_2^2}  \\
 \geq & \mathbb{E} \br{  \| \E_{Y} \hat{f} ({Y}_1 + \omega_1, {Y}_2+\omega_2, \ldots, {Y}_n+\omega_n)- f\|_2^2},
\end{split}
\end{align}
where the last inequality is a result of Jensen's inequality. Note that $\E_{Y} [\hat{f} ({Y}_1 + \omega_1, {Y}_2+\omega_2, \ldots, {Y}_n+\omega_n)]$ is a valid estimator that is only using information in $Y_1, Y_2, \ldots, Y_n$. 

Hence, for every estimator $\hat{f} (\tilde{Y}_1, \tilde{Y}_2, \ldots, \tilde{Y}_n)$ we have
\[
\sup_{f \in \Sigma_{\mathfrak{h}}(\beta, L)} \mathbb{E}  \| \E_{Y} \hat{f} ({Y}_1 + \omega_1, {Y}_2+\omega_2, \ldots, {Y}_n+\omega_n)- f\|_2^2 \leq \sup_{f \in \Sigma_{\mathfrak{h}}(\beta, L)} \mathbb{E} \|\hat{f} (\tilde{Y}_1, \tilde{Y}_2, \ldots, \tilde{Y}_n) - f\|_2^2
\]
Hence, 
\[
\inf_{\hat{\nu}} \sup_{f\in \Sigma_{\mathfrak{h}}(\beta, L)} \mathbb{E}  \| (\hat{\nu} ({Y}_1 {Y}_2, \ldots, {Y}_n))- f\|_2^2 \leq \inf_{\hat{\nu}} \sup_{f \in\Sigma_{\mathfrak{h}}(\beta, L)} \mathbb{E} \|\hat{\nu} (\tilde{Y}_1, \tilde{Y}_2, \ldots, \tilde{Y}_n) - f\|_2^2,
\]
which implies:

\[
R_2(\Sigma_{\mathfrak{h}}(\beta, L), \sigma_n) \leq R_2(\Sigma_{\mathfrak{h}}(\beta, L), \tilde{\sigma}_n). 
\]
\end{proof}

In our proofs, we will be using the properties of subGaussian and subexponential random variables. Recall that a variable $X$ is called \textit{sub-gaussian} with \textit{sub-gaussian norm} $\|X\|_{\rm subgau}$ if
\begin{align*}
    \|X\|_{\rm subgau}:=\inf\cb{t>0:\E\br{X^2/t^2}\le 2} 
\end{align*}
is bounded. Similarly, $X$ is called \textit{sub-exponential} with \textit{sub-exponential norm} $\|X\|_{\rm subexp}$ if
\begin{align*}
    \|X\|_{\rm subexp}:=\inf\cb{t>0:\E\br{|X|/t}\le 2}
\end{align*}
is bounded. The following lemma summarizes some of the useful properties of subGaussian and subexponential random variables.  
\begin{lemma}[Lemma 2.7.6, Lemma 2.7.7 and Exercise 2.7.10 of \cite{vershynin2018high}]\label{basic properties of subgau and subexp}The following is true for sub-gaussian and sub-exponential random variables
    \begin{enumerate}
    \item A standard normal Gaussian random variable $\xi\sim N(0,1)$ is sub-gaussian.
        \item Random variable $X$
 is sub-gaussian if and only if $X^2$ is sub-exponential, and if this the case, $\|X^2\|_{\rm subexp}=\|X\|_{\rm subgau}$.
 \item Let $X$ and $Y$ be
 sub-gaussian random variables. Then $XY$ is sub-exponential with $\|XY\|_{\rm subexp}\le \|X\|_{\rm subgau}\|Y\|_{\rm subgau}$
 \item Let $X$ be a sub-gaussian (resp. sub-exponential random variable), then there exists an absolute constant $C$ such that $\|X-\E X\|_{\rm subgau}\le C\|X\|_{\rm subgau}$ (resp. $\|X-\E X\|_{\rm subexp}\le C\|X\|_{\rm subexp}$)
    \end{enumerate}
\end{lemma}

The next two lemmas are two basic concentration results, one for simple quadratic functions of Gaussian random vectors, and the second one for the concentration of sum of subexponential random variables. 

\begin{lemma}\label{lM}\citep{laurentmassart} 
Let $Z_1, Z_2, \ldots, Z_n$ be iid random variables distributed as $\mathcal{N}(0,1)$ and $a \in \mathbb{R}^n_{\geq 0}$. Then for all $t\geq 0$,

$$\mathbb{P}\left(\sum_i a_i Z^2_i- \sum_i a_i \geq 2\|a\|_2\sqrt{t} + 2\|a\|_\infty t \right) \leq e^{-t},$$

and
$$\mathbb{P}\left(\sum_i a_i Z^2_i- \sum_i a_i \leq -2\|a\|_2\sqrt{t} \right) \leq e^{-t}.$$
\end{lemma}

\begin{lemma}\label{berstein's inequality for sub-exponential distribution}[Berstein's inequality for sub-exponential distributions. Theorem 2.8.1, \cite{vershynin2018high}]
Let $Z_1,...,Z_n$ be independent mean zero, sub-exponential random variables. Then, for $t>0$, we have
\begin{equation}
\mathbb{P}\cb{\abs{\sum_{j=1}^nZ_j}\ge t}\le \exp\br{-c\min\pr{\frac{t^2}{\sum_{j=1}^n \nm{Z_j}_{\rm subexp}^2},\frac{t}{\max_{1
\le j\le n}\nm{Z_j}_{\rm subexp}}}}
\end{equation}
where $c>0$ is an absolute constant.
\end{lemma}

The next lemma is a version of the well-known Berry-Esseen theorem. 

\begin{lemma}\label{BE}(non-iid Berry-Esseen theorem \citep{Ber-Ess})
Let $Z_1, Z_2, \ldots, Z_n$ be independent, zero mean random variables with finite second moment. Suppose that there exists a fixed number $\kappa$ such that  
$$\max_{1\leq i \leq n}\frac{\mathbb{E}|Z_i|^3}{\mathbb{E}(Z_i)^2} \leq \kappa.$$
 If $F, \Phi$ denote the CDFs of $\frac{\sum_i Z_{i}}{\sqrt{\sum_i \mathbb{E} Z_{i}^{2}}}$ and $\mathcal{N}(0,1)$ respectively, then there exists a constant $C_0$, such that
$$\sup _{x}\left|F(x)-\Phi(x)\right|\leq C_0 \frac{\kappa}{\sqrt{\sum_i \mathbb{E} Z_{i}^{2}}}.$$
\end{lemma}

Our last lemma will be used in the proofs of the lower bounds.

\begin{lemma}\label{GV}(Gilbert-Varshamov bound \citep{Gilbert},\citep{Varshamov}, \citep{Mainbook}) 
Let $\Omega=\{0,1\}^m$, where $m \geq 8$. Then, there exists a hypercube  $\Omega^{\prime}:=\left\{\omega^1, \ldots, \omega^M\right\} \subseteq \Omega$ such that
$M \geq 2^{m / 8}$, each $\omega_i$ has at least $m/16$ nonzero entries, and we have
$\rho\left(\omega^j, \omega^k\right) \geq m / 16$ for each $j \neq k$.
\end{lemma}
%Finally, we restate a helpful lemma (\cite{vershynin2018high}, Lemma 6.2.2).
%
%
%\begin{lemma}\label{chaos} Let $X, X^{\prime} \overset{iid}{\sim}  N\left(0, I_n\right)$ and let $A=\left(a_{i j}\right)$ be an $n \times n$ matrix. Then
$$
%\mathbb{E} \exp \left(\theta X^{T} A X^{\prime}\right) \leq \exp \left(B \theta^2\|A\|_F^2\right)
$$
%for some constant $B>0$ and all $\theta$ satisfying $|\theta| \leq \frac{1}{\sqrt{B}\|A\|}$.
%\end{lemma}
%\textcolor{red}{I am not sure when the name "Gaussian Chaos" has come from. It does not make sense to me. Make sure it it correctly cited. Also, please define $B$. Also, try to specify $b$ as much as you can. I think it should be possible to say exactly what $b$ is. After all these are Gaussian integrals and have explicit forms.  }
%
%

\subsection{Auxiliary lemmas}

In our main proofs, we need a few properties of the functions in $\Sigma(\beta, L)$ that we prove below.

\begin{lemma}\label{lem: all derivatives}
    Suppose $|f| \le 1$ and $f\in \Sigma(\beta, L)$. Then there exists a constant $c_{\beta, L}>0$ (independent of $f$) such that for all $1\le r\le k$, where $k$ is the largest integer strictly less than $\beta$ and $x\in [0,1]$, $|f^{(r)}(x)|\le c_{\beta, L}$. 
\end{lemma}

\begin{proof}
The first step of the proof is to establish the lemma for the case $r = k$. To this end, we first show that there exists a constant $c_{\beta} > 0$, depending only on $\beta$, such that for any function $f \in \Sigma(\beta, L)$ with $|f| \le 1$, there exists a point $\eta_k \in [0, 1]$ satisfying $|f^{(k)}(\eta_k)| \le c_{\beta}$. Since the proof in the general case involves more intricate notation, we provide detailed arguments for the cases $k = 2$ and $k = 3$, and then present a proof sketch for the general case.

The case $k=1$ is immediate from Lagrange's mean value theorem. For $k=2$, we consider points $a_1=0, a_2=\frac{1}{3}, a_3=\frac{2}{3}, a_4=1$ and apply the mean value theorem twice to quotient
\begin{align*}
    \abs{\frac{\frac{f(a_1)-f(a_2)}{a_1-a_2} - \frac{f(a_3)-f(a_4)}{a_3-a_4}}{a_2-a_3}}.
\end{align*}
By the mean value theorem, $\frac{f(a_1)-f(a_2)}{a_1-a_2}=f'(\xi_{12})$ for some $\xi_{12}\in [a_1,a_2]$ and $\frac{f(a_3)-f(a_4)}{a_3-a_4}=f'(\xi_{34})$ for some  $\xi_{34}\in [a_3,a_4]$. Since $|f|\le 1$, we have 
\begin{align*}
    |f'(\xi_{12})|\le \frac{|f(a_1)|+|f(a_2)|}{|a_1-a_2|}\le \frac{2}{1/3}=6
\end{align*}
and similarly $|f'(\xi_{34})|\le 6$.

Note that $|\xi_{12}-\xi_{34}|\ge |a_2-a_3|= \frac{1}{3}$. We have that this quotient can be lower-bounded by
\begin{align*}
    \abs{\frac{f'(\xi_{12})-f'(\xi_{34})}{a_2-a_3}}= \abs{\frac{f'(\xi_{12})-f'(\xi_{34})}{\xi_{12}-\xi_{34}}\frac{\xi_{12}-\xi_{34}}{a_2-a_3}}\ge\abs{f''(\xi_{1234})},
\end{align*}
for some $\xi_{1234}\in [0,1]$. It follows that 
\begin{align*}
\abs{f''(\xi_{1234})}\le \frac{6 \times 2}{1/3}=36.    
\end{align*}
Hence, by setting $c_\beta= 36$, this shows that for $k=2$, any for any $f \in \Sigma(\beta, L)$, there exists $\eta_2$ such that $\abs{f''(\eta_2)} \leq c_\beta$. 

To prove a similar result for $k=3$, we have to construct more intervals. Hence, in this case we consider  $a_1=0, a_2=\frac{1}{9}, a_3=\frac{2}{9},a_4=\frac{3}{9},a_5=\frac{6}{9},a_6=\frac{7}{9},a_7=\frac{8}{9}$, and $a_8=1$. To be able to find an upper bound for the third derivative, we consider the following ratios:
\begin{align*}
\left|\frac{\frac{f(a_4)- f(a_3)}{a_4-a_3} - \frac{f(a_2)-f(a_1)}{a_2-a_1} }{a_3-a_2} \right| \quad \text{and} \quad \left|\frac{\frac{f(a_8)- f(a_7)}{a_8-a_7} - \frac{f(a_6)- f(a_5)}{a_6-a_5}}{a_7-a_6}\right|.
\end{align*}

Using an argument similar to the one we used for the case $k=2$, we can prove that there exists $\xi_{1234}\in [a_1,a_4]$ such that  $|f''(\xi_{1234})|\le 3^2\cdot 3^2 \cdot 4=324$. Similarly, there exists $\xi_{5678}\in [a_5,a_8]$, such that $|f''(\xi_{5678})|\le 3^2\cdot 3^2 \cdot 4=324$. Notice that by construction, $|\xi_{1234}- \xi_{5678}|\ge |a_4-a_5|\ge \frac{1}{3}$. Hence, we have that there exists $\xi_{12345678}\in [\xi_{1234},\xi_{5678}]$ such that 
\begin{align}\label{eq:thirdtermbound}
\left| \frac{f''(\xi_{5678}) - f''(\xi_{1234})}{a_5-a_4} \right| \geq \left| \frac{f''(\xi_{5678}) - f''(\xi_{1234})}{\xi_{5678}-\xi_{1234}}\right| =|f'''(\xi_{12345678})|, 
\end{align}
where the last equality is the result of the mean value theorem. 

Using \eqref{eq:thirdtermbound}, it is straightforward to prove that: 
\begin{align*}
|f'''(\xi_{12345678})|\le 324 \times 2 \times 3=2^3\times 3^3.
\end{align*}
Hence, in the case $k=3$, if we set $c_{\beta, L} = 2^3\times 3^3$, we have shown the existence of $\eta_3$ for which 
$|f'''(\eta_3)| \leq c_{\beta, L}$. Note that the bound is solely determined by our division of $[0,1]$ and the only properties we have used about $f$ are the existence of its derivatives and and the uniform bound $|f|\le 1$. 

For the general $k$-th derivative case, let $a_1<\cdots<a_{2^k}$ be the end points of the intervals in the $k$-th step construction of the Cantor set. Applying the mean value theorem inductively on these end points like in the case of $k=2$ and $k=3$ and observe that any two end points are at least $\frac{1}{3^k}$ apart, one can show there exist $c_{\beta}>0$ ($c_{\beta}$ is a constant that only depends on $\beta$, and hence $k$) such that  for every $f \in \Sigma(\beta, L)$, there exists $\eta_k$ such that $|f^{(k)}(\eta_k)|\le c_{\beta}$. 

So far, we have only proved the boundedness of $f^k(\cdot)$ at a particular point we called $\eta_k$ that can depend on the choice of $f$. The second step of the proof of \cref{lem: all derivatives} is to use this fact to show that $f^k(\cdot)$ is uniformly bounded by a constant independent of the choice of $f$. Towards this goal, note that since $f\in \Sigma(\beta, L)$ we have
\begin{align}
    |f^{(k)}(x)-f^{(k)}(\eta_k)|\le L|x-\eta_k|^{\beta-k}\le L
\end{align}
and thus $|f^{(k)}(x)|\le L+c_{\beta}$ for any $x\in [0,1]$.

So far we have proved that $|f^{(k)}(x)| \leq c_{\beta, L}$. Now we have to prove the result for $r= k-1, k-2, \ldots, 1$. Note that since we have proved the boundedness of $|f^{(k)}(x)|$ we can conclude by using the mean value theorem that:
\begin{align*}
    |f^{(k-1)}(x) - f^{(k-1)}(y)| \leq c_{\beta, L} |x-y|. 
\end{align*}
Now, to prove that there exists $c_{\beta, L}$ such  that $|f^{(k-1)}(x)| \leq c_{\beta, L}$, we follow the same steps as the ones we followed for bounding $f^{(k)}(\cdot)$: 
\begin{enumerate}
\item Proving that there exists $c_{\beta, L}$ such that for every $f \in \Sigma(\beta, L)$ there exists an $\eta$ for which $|f^{(k-1)}(\eta)| \leq c_{\beta,L}$.

\item We use the fact that $|f^{(k-1)}(x) - f^{(k-1)}(y)| \leq c_{\beta, L} |x-y|$ in conjunction with the Step 1 to prove the uniform boundedness of  $f^{(k-1)}(x)$. 
\end{enumerate}
The argument above shall be used inside an inductive argument to finish the proof for  $r=k-2,...,1$.  
  
\end{proof}

\begin{lemma}\label{lem: product is holder}
There exists a constant $C_{\beta, L}$, such that for all functions $f$ and $g$ that satisfy $|f|\le 1$, $|g|\le 1$ and $f,g\in \Sigma(\beta, L)$, we have $fg\in \Sigma(\beta, C_{\beta, L})$.  
\end{lemma}

\vspace{3mm}
\begin{proof}
 By Leibnitz's rule for derivatives, we have
\begin{align*}
    (fg)^{(k)}(x)=\sum_{r=0}^k {k\choose r}f^{(k-r)}(x) g^{(r)}(x)
\end{align*}
Now for $1\le r\le k-1$
\begin{align*}
\begin{split}
        & \abs{f^{(k-r)}(x) g^{(r)}(x)-f^{(k-r)}(y) g^{(r)}(y)}  \\
\le & \abs{f^{(k-r)}(x) g^{(r)}(x)-f^{(k-r)}(y) g^{(r)}(x)}+\abs{f^{(k-r)}(y) g^{(r)}(x)-f^{(k-r)}(y) g^{(r)}(y)}  \\
\overset{(a)}{\le} & c_{\beta, L}\abs{f^{(k-r)}(x) -f^{(k-r)}(y) }+c_{\beta, L} \abs{ g^{(r)}(x)- g^{(r)}(y)}   \\
\overset{(b)}{\le} & 2c_{\beta,L}^2|x-y|\le 2c_{\beta,L}^2 |x-y|^{\beta-k},
\end{split}
\end{align*}
where to obtain Inequality (a), we have used \cref{lem: all derivatives} that shows $|g^{(r)}(x)|< c_{\beta, L}$ and $|f^{(k-r)}(y)| < c_{\beta, L}$. To obtain the inequality (b) we have used the mean value theorem combined with \cref{lem: all derivatives}. To obtain the last inequality we have used the fact that $0\leq\beta-k<1$. Next we study $f^{(k)}(x)g(x)$. 

\begin{align*}
    &\abs{f^{(k)}(x)g(x)-f^{(k)}(y)g(y)} \\
\le & \abs{f^{(k)}(x) g(x)-f^{(k)}(y) g(x)}+\abs{f^{(k)}(y) g(x)-f^{(k)}(y) g(y)}  \\
\le & 2L c_{\beta, L} |x-y|^{\beta-k}.
\end{align*}
Similarly, we have
\begin{align*}
    &\abs{g^{(k)}(x)f(x)-g^{(k)}(y)f(y)}
\le 2L|x-y|^{\beta-k}.
\end{align*}
Setting $C_{\beta,L}:=2^k c_{\beta,L}^2 + 2c_{\beta, L}L$ and using triangular inequality complete the proof.
\end{proof}

Another component we will need in our main proofs, is the following definition and construction borrowed from Section 2.4 of \cite{Mainbook}.

\begin{defn}\label{def:basicfunction}
For the function class $\Sigma(\beta, L)$, we call a function $\phi: \mathbb{R} \rightarrow \mathbb{R}$ a \emph{basic function} if and only if it satisfies the following four properties:

(a) $\phi$ is infinitely differentiable on $\mathbb{R}$.

(b) $\phi(x)=0$ if $x \notin\left(-\frac{1}{2}, \frac{1}{2}\right)$.

(c) $\phi(x)>0$ if $x \in\left(-\frac{1}{2}, \frac{1}{2}\right)$.

(d) $\max _x\left|\phi^{(k+1)}(x)\right| \leq L$, where $k$ is the largest integer strictly less than $\beta$.
\end{defn}

\vspace{5mm}
Now we construct an explicit basic function for $\Sigma(\beta, L/2)$ that will be used throughout this paper. Consider 
\begin{equation}
    \phi_0(x):=\left\{\begin{array}{cc}
\exp \left(\frac{4}{4x^2 -1}\right), & |x| < \frac{1}{2}, \\
0, & \text{otherwise}.
\end{array}\right.
\end{equation}
It is easy to see that $\phi_0$ satisfies the conditions (a), (b) and (c). If we define the function
\begin{align}\label{explicit construction of basic function}
    \phi(x):=\frac{L}{2} \phi_0(x)/ \overline{\phi_0},
\end{align}
with $\overline{\phi_0} := \max_x |\phi_0^{(k+1)}(x)|$,  then $\phi(x)$  will  satisfy condition (d) as well. Hence, $\phi$ is a basic function for $\Sigma(\beta, L/2)$.

\subsection{Proof of the lower bound  of \cref{l_2}}\label{roadmap:lb}

\quad

For the lower bound, we need to show $ \sup_{f \in \Sigma_{\mathfrak{h}}(\beta, L)} \mathbb{E}_f \|f- \hat{\nu}\|_2$ is large for all estimators $\hat{\nu}$. It is clear that $$ \max_{\nu \in \{\nu_0,\nu_1, \ldots \nu_M\} \subseteq \Sigma_{\mathfrak{h}}(\beta, L)} \mathbb{E}_{\nu} \|\nu - \hat{\nu}\|_2 \leq \sup_{f \in \Sigma_{\mathfrak{h}}(\beta, L)} \mathbb{E}_f \|f- \hat{\nu}\|_2,$$
for any $\{\nu_0, \nu_1, \ldots \nu_M\} \subseteq \Sigma_{\mathfrak{h}}(\beta, L)$. Therefore, it is sufficient to construct $M$ functions, $\nu_0, \nu_1, \ldots, \nu_M \in \Sigma_{\mathfrak{h}}(\beta, L)$ such that for all estimators $\hat{\nu}$, $\mathbb{E}_{\nu_j} \|\nu_j - \hat{\nu}\|_2$ is large for at least one $\nu_j$. In order to obtain a lower bound for $\mathbb{E}_{\nu_j} \|\nu_j - \hat{\nu}\|_2$, we can use Markov inequality and obtain
\[
\mathbb{E} \|\nu_j -\hat{\nu}\|_2 \geq t \cdot \mathbb{P}(\|\nu_j- \hat{\nu}\|_2>t ),
\]
which holds for every $t>0$. Hence, it suffices to construct functions $\nu_0, \nu_1,..., \nu_M$ in $\Sigma_{\mathfrak{h}}(\beta, L)$ such that for all estimators $\hat{\nu}$, there exists at least one $\nu_j$ for which $ \mathbb{P}(\|\nu_j- \hat{\nu}\|_2>t )$   is bounded away from zero. It is straightforward to see that the larger $\mathbb{P}_{\nu_j}\left( \|\nu_j - \hat{\nu}\|_2 > t \right) $ is, the larger the lower bound will be. Hence the problem of finding a good lower bound reduces to the problem of choosing $M$ and $\nu_0, \nu_1, \nu_2, \ldots, \nu_M$. 

Our next lemma which is a restatement of Theorem 2.5.3 in \cite{Mainbook}, clarifies what properties we expect the functions to satisfy, and what lower bound we can obtain for $\max_{0\leq j\leq M}\mathbb{P}_{\nu_j}(\|\hat{\nu}- \nu_j\|_2\geq t)$. Recall that the notation of likelihood ratio $\Lambda(\nu_0, \nu_1)$ is defined in \cref{lr:def}.

\begin{lemma}\label{thm1}
    Suppose $\nu_0 , \nu_1, \ldots, \nu_M$ are real-valued functions such that the following conditions hold:
    
    (i) For all $0 \leq j \neq k \leq M$, $d(\nu_j,\nu_k) \geq 2 s_n >0$, where $d$ is a distance on $\Sigma(\beta, L)$. (For our current proof, assume $d$ is given by $L_2$ norm.) 
    
    (ii) For all $j \in \{1,2,\ldots, M\}$, there exists $\lambda < 1$ such that $\lambda_j < \lambda$ and we have $$\mathbb{P}_{\nu_j}(\Lambda(\nu_0, \nu_j) > M^{-\lambda_j}) > p^* ,$$ where $p^*$ is independent of $n$, $j$, and $\Lambda (\nu_0, \nu_j)$ is defined in \cref{lr:def}. Then, for every estimator $\hat{\nu}_n$, we have 
    $$\max_{0\leq j\leq M}\mathbb{P}_{\nu_j}(d(\hat{\nu}_n, \nu_j)\geq s_n)\geq p^*/2.$$
\end{lemma}

So, one of the main challenges of the proof is to construct the suitable choices of functions $\nu_0, \nu_1, \ldots, \nu_M$ that satisfy the assumptions in \cref{thm1}. To construct these functions, we use $\phi$ which is a basic function (see Definition \ref{def:basicfunction}) for  $\Sigma(\beta, L/2)$. Set
\begin{equation}\label{choice of delta}
    \delta_n = \pr{\frac{\max\pr{1,\sigma_n^4}}{n}}^{1 \over 2\beta+1} \quad \text{and} \quad m=[\delta_n^{-1}], 
\end{equation}
and define the bump functions
\begin{align}
f_{j,n}(x):=\delta_n^{\beta} \phi\left(\frac{x-b_j}{\delta_n}\right), \quad j=1, \ldots, m,
\end{align}
where
\begin{align}
 b_{j}=(2 j-1) \frac{\delta_n}{2}.  
\end{align}
Clearly, $f_{j,n}$'s have disjoint supports. Further more, we have

\begin{lemma}[Proposition 2.4.3 of \cite{Mainbook}]\label{f and f^2 holder}
  For $n\ge 1$, we have $f_{j,n}\in \Sigma(\beta,L/2)$.
\end{lemma}

\iffalse
\begin{proof}[Proof of \cref{f and f^2 holder}]

\textcolor{red}{I do not like the way the following proof is written. Please consider the different cases, and prove it carefully. } \textcolor{blue}{This proof is actually in the book, say page 93 of \cite{Nonp}. Eventually should be removed.}
    Recall the definition of $\phi_0$ as defined in \cref{explicit construction of basic function} and $\phi(x)=\frac{L}{2} \phi_0(x)/ \overline{\phi_0}$, with $\overline{\phi_0} := \max_x |\phi_0^{(k+1)}(x)|$. Then 
    \begin{align}
        \begin{split}
         & \abs{\frac{d^k}{dx^k}{f_{i,n}}(x)-\frac{d^k}{dy^k}{f_{j,n}}(y)}
        \le \frac{L}{2}\cdot  \de_n^{\beta-k}\cdot \abs{\frac{x-b_j}{\de_n}-\frac{y-b_j}{\de_n}}^{\beta-k}
        = \frac{L}{2} \abs{x-y}^{\beta-k}.
        \end{split}
    \end{align}
    This completes the proof of the lemma.
\end{proof}
\fi

 By combining these bump functions, we aim to construct the functions that are required in \cref{thm1}. Towards this goal, for $l=1, \ldots, 2^{m}$, consider $\omega_{l} \in\{0,1\}^{m}$ and define 
\begin{equation}\label{eq:g_def}
g_{l}(x):=1-\sum_{j=1}^{m}\left(w_{l}\right)_{j} f_{j,n}(x).
\end{equation}
In order to construct $\nu_0 , \nu_1, \ldots, \nu_M$, we will pick $M$ of these functions that satisfy certain properties that we will clarify later. We first observe that
\begin{lemma}\label{belonging}
For sufficiently large $n$ and $l=1, \ldots, 2^{m}$, $g_l \in  \Sigma_{\mathfrak{h}}(\beta, L)$.
\end{lemma}

The proof of this lemma is presented in \cref{proofl_2}. Now, define $\nu_{0} \equiv 1$ and $\nu_{1}, \ldots, \nu_{M}$ to be elements in $\left\{g_{l}: l=1, \ldots, 2^{m}\right\}$ such that for the corresponding elements in $\{0,1\}^{m}$, denoted $\omega_{1}, \ldots, \omega_{M}$, we have $\rho\left(\omega_{i}, \omega_{j}\right) \geq \frac{m}{16}$ for all $i, j \in\{1, \ldots, M\}$, and each $\omega_i$ has at least $m/16$ nonzero entries.
We fix $M$ to be $[2^{m/8}]$ which is guaranteed by \cref{GV}. It is straightforward to confirm that for all $0 \leq j \neq k \leq M$, $$ \|\nu_j-\nu_k\|_{2} \geq \sqrt{\frac{m}{16}}\|f_{1,n}\|_2 = \sqrt{\frac{m}{16}} \delta_n^{\beta+1/2}\|\phi\|_2 = \frac{\delta_n^{\beta} \|\phi\|_2}{4}.$$

The likelihood ratio required in  \cref{thm1} is given by (for simplicity we denote $\prod_{i}=\prod_{i}^n$ and $\sum_i=\sum_{i=1}^n$)
$$
\begin{aligned}
& \Lambda\left(\nu_{0}, \nu_{l}\right)=\prod_{i} \sqrt{\frac{{\sigma}^2_n +  \nu_{l}^{2}\left(i/n\right)}{1+ {\sigma}^2_n}} \exp \left[\sum_i \frac{y_{i}^{2}}{2 ({\sigma}^2_n + \nu_{l}^{2}\left(i/n\right))}-\sum_i \frac{y_{i}^{2}}{2(1+{\sigma}^2_n)}\right] \\
& =\exp \left[\sum_{i} \frac{y_{i}^{2}}{2}\left( \frac{1}{{\sigma}^2_n + \nu_{l}^{2}\left(i/n\right)}- \frac{1}{1+ {\sigma}^2_n}\right)+\frac{1}{2}\sum_{i} \log \frac{\sigma_n^2+\nu_{l}^{2}\left(i/n\right)}{\sigma_n^2+1}\right]\\
& =\exp \left[\sum_{i} \frac{y_{i}^{2}}{2(1 + {\sigma}^2_n)}\left( \frac{\sigma_n^2+1}{\sigma_n^2+\nu_{l}^{2}\left(i/n\right)}- 1\right)+\frac{1}{2}\sum_{i} \log \frac{\sigma_n^2+\nu_{l}^{2}\left(i/n\right)}{\sigma_n^2+1}\right].
\end{aligned}
$$

To apply \cref{thm1}, we need to show the following result:
% \textcolor{red}{So far we have kept the choice of $\alpha$ general. We have to fix it before this lemma. }

\begin{lemma}\label{lem:likelihoodratio}
There exists some $\lambda< 1$, such that for all $l$,
$$\mathbb{P}_{\nu_{l}}\left(\Lambda\left(\nu_{0}, \nu_{l}\right)> M^{-\lambda_l}\right)>p^{*},$$
where $\lambda_l< \lambda$ and the functions $\nu_l$ are defined earlier.
%\[
%\nu_l = 1- \sum_{j=1}^m (\omega_l)_j f_{j,n},
%\]
%where $\omega_l \in \{0,1\}^m$. Furthermore, assume that $\|\nu_0- \nu_\ell\|\geq \frac{\delta_n^{\beta}}{4}$.
%Then, 
%$$\mathbb{P}_{\nu_{l}}\left(\Lambda\left(\nu_{0}, \nu_{l}\right)> M^{-\lambda_l}\right)>p^{*},$$ and $\lambda_l< \lambda$ for some $\lambda< 1$.
\end{lemma}
The proof of this lemma is postponed to \cref{proofl_2}.
Combining the results of  \cref{thm1} and \cref{lem:likelihoodratio} gives us
$$
\begin{aligned}
& \max_{0\leq j\leq M}\mathbb{P}_{\nu_j}\left(\|\hat{\nu}- \nu_j\|_2\geq \frac{ \delta_n^{\beta}\|\phi \|_2}{8} \right) = \max_{0\leq j\leq M}\mathbb{P}_{\nu_j}\left(\|\hat{\nu}- \nu_j\|^2_2\geq \frac{ \delta_n^{2\beta}\|\phi \|^2_2}{64} \right) \geq p^*/2  .\end{aligned}$$
Using Markov's inequality, for any estimator $\hat{\nu}$,

\begin{align*}
& \max_{0\leq j\leq M}\mathbb{E}_{\nu_j}\|\hat{\nu}- \nu_j\|^2_2 \geq \frac{p^* \delta_n^{2\beta}\|\phi \|^2_2}{128}
\end{align*}
Hence, in the case when $\sigma_n^4=o(n^4)$, we have
\begin{align*}
   R_2(\Sigma_{\mathfrak{h}}(\beta, L), {\sigma}_n)\geq \frac{p^* \delta_n^{2\beta}\|\phi \|^2_2}{128} = \Omega\left(\pr{\max \pr{1,\sigma_n^4} \over n}^{\frac{2\beta}{2\beta+1}}\right).
 \end{align*}
 If $\sigma_n^4=\Omega(n)$, then using the same argument as in \cref{proof of the lower bound for the classical model in L_2 norm}, we have by \cref{lem:risk:monotonicity} that
\begin{align*}
   R_2(\Sigma_{\mathfrak{h}}(\beta, L), {\sigma}_n)=\Omega(1).
 \end{align*}
 This completes the proof of the lower bound.

\subsubsection{Proof of \cref{belonging} and \cref{lem:likelihoodratio}}\label{proofl_2}

\begin{proof}(\cref{belonging})
Recall
\begin{equation*}
g_{l}(x)=1-\sum_{j=1}^{m}\left(w_{l}\right)_{j} f_{j,n}(x).
\end{equation*}
where $f_{j,n}(x):=\delta_n^{\beta} \phi\left(\frac{x-b_j}{\delta_n}\right)$. Note that according to \cref{f and f^2 holder}  $f_{j,n}(x) \in \Sigma(\beta,L/2)$. Since $f_{j,n}$'s have disjoint supports, for each $x\in [0,1]$
\begin{align*}
    g_{l}(x)
    = & 1-\sum_{j=1}^{m}\left(w_{l}\right)_{j} f_{j,n}(x)
    \ge 1- \delta_n^{\beta} \phi^*,
\end{align*}
where $\phi^* = \max_{z \in [0,1]} \phi(z)$.
Therefore for large enough $n$, $g_l\in [\mathfrak{h},1]$. Let $k$ be the largest integer strictly less than $\beta$. Then according to \cref{f and f^2 holder} we have
    \begin{align*}
        |f_{j,n}^{(k)}(x)-f_{j,n}^{(k)}(y)|\le \frac{L}{2}|x-y|^{\beta-k}.
    \end{align*}

    Now for any $x,y\in [0,1]$, let $j_x,j_y\in \{1,2,...,m\}$ such that $x\in \text{supp}(f_{j_x,n})$ and $y\in \text{supp}(f_{j_y,n})$. If $j_x\ne j_y$, since the supports of $f_{j_x,n}$ and  $f_{j_y,n}$ are disjoint, by the triangular inequality we have that 
\begin{equation*}
    |g_l^{(k)}(x)-g_l^{(k)}(y)|\le |f_{j_x,n}^{(k)}(x)-f_{j_x,n}^{(k)}(y)|+|f_{j_y,n}^{(k)}(x)-f_{j_y,n}^{(k)}(y)|\le L|x-y|^{\beta-k}.
\end{equation*}
If $j_x=j_y$, then
\begin{equation*}
    |g_l^{(k)}(x)-g_l^{(k)}(y)|\le |f_{j_x,n}^{(k)}(x)-f_{j_x,n}^{(k)}(y)|\le \frac{L}{2}|x-y|^{\beta-k} \le L|x-y|^{\beta-k}.
\end{equation*}
\end{proof}
\vspace{.2cm}

\begin{proof}(\cref{lem:likelihoodratio})
Our goal is to prove that for random variables
\begin{equation}
Y_i=Y_{l,i}:=\nu_l(i/n)\xi_i+\tau_i,~i=1,2,...,n
\end{equation}
we have
\begin{align}\label{eq:likelihood}
\begin{split}
    & \mathbb{P}_{\nu_l}\left(\sum_{i} \frac{Y_{i}^{2}}{1 + {\sigma}^{2}_n}\left( \frac{\sigma_n^2+1}{\sigma_n^2+\nu_{l}^{2}\left(i/n\right)}- 1\right)+\sum_{i} \log \frac{\sigma_n^2+\nu_{l}^{2}\left(i/n\right)}{\sigma_n^2+1}>-2\lambda_{l} \log M\right)  \\
& >p^{*},
\end{split}
\end{align}

for all $l=1, \ldots, M$. If we define $s_{l,i}:=\frac{1-\nu_l^2(i/n)}{{\sigma}^{2}_n +1}$, then proving \cref{eq:likelihood} is equivalent to proving that for some $p^*>0$,
\begin{align}\label{eq:likelihood3}
& \mathbb{P}_{\nu_l} \left[\sum_{i} \frac{Y_{i}^{2}}{1+ {\sigma}^{2}_n}\left( \frac{1}{1 - s_{l,i}}- 1\right)+\sum_{i} \log (1- s_{l,i})>-2\lambda_{l} \log M \right]>p^{*}.
\end{align}
We aim to simplify the likelihood expression further by using the mean value theorem. According to this theorem, we have
that for some $\epsilon_{l,i}, \epsilon_{l,i}^{\prime} \in\left(0, s_{l,i}\right),$
$$
\begin{aligned}
& \frac{1}{1-s_{l,i}}=1+s_{l,i}+\frac{s_{l,i}^{2}}{\left(1-\epsilon_{l,i}\right)^{3}}\\
& \log (1-s_{l,i})=-s_{l,i}-s_{l,i}^{2} \frac{1}{2\left(1-\epsilon_{l,i}^{\prime}\right)^{2}}.
\end{aligned}$$
Using these two equations we conclude that proving \cref{eq:likelihood3} is equivalent to proving:

\begin{align}\label{eq:likelihood4}
\begin{split}
    & \mathbb{P}_{\nu_l}\left[\sum_{i} \frac{Y_{i}^{2}}{1+ {\sigma}^{2}_n}\left(s_{l,i}+\frac{ s_{l,{i}}^{2}}{\left(1-\epsilon_{l,i}\right)^{3}}\right)-\sum_{i}\left(s_{l,i}+\frac{s_{l,i}^{2}}{2\left(1-\epsilon_{l,i}^{\prime}\right)^{2}}\right) >-2\lambda_{l} \log M\right]
\end{split}
\end{align}

Since $\delta_n \to 0$ as $n\to \infty$, for large enough $n$, we have that  $0< \epsilon_{l,i}, \epsilon_{l,i}^{\prime} < s_{l,i} < 2 [1-\nu_l(i/n)]\leq 2 \delta_n^{\beta} \phi^{*}<\frac{1}{2}$ (Notice that for one and only one value of $j=1, \ldots, m$, we have $f_{j,n}(i / n) = \delta_n^{\beta} \phi\left(\frac{i/n-b_j}{\delta_n}\right)>0$, since $f_{j,n}$'s have disjoint support by definition). Hence, in order to prove \cref{eq:likelihood4} it is sufficient to prove the following:
\begin{align*}
& \mathbb{P}_{\nu_l}\left[\sum_{i} \frac{Y_{i}^{2}}{1+ {\sigma}^{2}_n}\left(s_{l,i}+ s_{l,i}^{2}\right)-\sum_{i}\left(s_{l,i}+2 s_{l,i}^{2}\right)>-2\lambda_{l} \log M\right]>p^{*}  
\end{align*}
which is equivalent to proving:
\begin{align}\label{eq:likelihood5}
\mathbb{P}_{\nu_l}\left[\sum_{i} s_{l,i}\left(\frac{Y_{i}^{2}}{1+ {\sigma}^{2}_n}-1\right)+\sum_i s_{l,i}^{2} \left( \frac{Y_{i}^{2}}{1+ {\sigma}^{2}_n}-2\right)>-2\lambda_{l} \log M\right]>p^{*} . 
\end{align}

Using \cref{lM}, for any $t \geq 0$ we have,

$$
\begin{aligned}
&  \mathbb{P}_{\nu_l}\left(\sum_i Y_{i}^{2}\left(\frac{1}{{\sigma}^{2}_n+\nu_{l}^2{(i / n)}}\right)  s_{l,i}^{2} \frac{\sigma_n^2+\nu_{l}^{2}(i/n)}{\sigma_n^2+1} \leq \sum_{i} s_{l,i}^{2} \frac{\sigma_n^2+\nu_{l}^{2}(i/n)}{\sigma_n^2+1} \right.\\
& \left. -2 \sqrt{t}  \sqrt{\sum_{i}  s_{l,i}^{4} \left(\frac{\sigma_n^2+\nu_{l}^{2}(i/n)}{\sigma_n^2+1}\right)^2}\right)\leq e^{-t}.
\end{aligned}
$$

Recalling the definition of $s_{l,i}$, we know $1-s_{l,i}=\frac{\sigma_n^2+\nu_{l}^{2}(i/n)}{\sigma_n^2+1}$. 
Choosing $t=\log 3$, we have
\begin{align}\label{eq:quadratic_terms} 
 \begin{split}
&\mathbb{P}_{\nu_l}\left(\sum_i Y_{i}^{2}\left(\frac{1}{{\sigma}^{2}_n+\nu_{l}^2{(i / n)}}\right)  s_{l,i}^{2} (1-s_{l,i}) \leq \sum_{i} s_{l,i}^{2} (1-s_{l,i}) \right. \\
& \left. -2 \sqrt{\log 3} \sqrt{\sum_{i}  s_{l,i}^{4} (1-s_{l,i})^2}\right)\leq \frac{1}{3}.
 \end{split}
\end{align}
Combining \cref{eq:quadratic_terms} in \cref{eq:likelihood5}, we conclude that a sufficient condition to prove \cref{eq:likelihood5} is to show that for some $p^*\geq0$,

\begin{align}\label{eq:likelihood6}
\begin{split}
    & \mathbb{P}_{\nu_l}\left(\sum_{i} s_{l,i}\left(\frac{Y_{i}^{2}}{1+ {\sigma}^{2}_n}-1\right)- \sum_{i} s_{l,i}^{2} (1+s_{l,i})-2 \sqrt{\log 3} \sqrt{\sum_{i}  s_{l,i}^{4} } >-2\lambda_{l} \log M \right)  \\
&>p^{*}+\frac{1}{3} .
\end{split}
\end{align}

The terms with a factor $f_{j,n}(i/n) f_{k,n}(i/n)$ are zero for $j \neq k$. Hence, we can conclude that for $\forall r \in \mathbb{Z}^{+}$, we have 
$$
\begin{aligned}
&\sum_{i}\pr{1- \nu_l(i/n)}^{r}=\sum_{i}\left[\sum_{j=1}^{m}\left(\omega_{l}\right)_j f_{j,n}(i / n)\right]^{r}=\sum_{i} \sum_{j=1}^{m}\left(\omega_{l}\right)_{j}^{r} f_{j,n}^r(i/n).
\end{aligned}
$$
Therefore,
\begin{align}\label{eq:moments_f}
\begin{split}
    &\sum_{i} \sum_{\{j: (\omega_l)_j=1\}} f_{j,n}^r(i/n) = \sum_{i} \sum_{\{j: (\omega_l)_j=1\}} \delta_n^{r\beta} \phi^r\pr{\frac{i/n-b_j}{\delta_n}}  \\
&= \sum_{\{j: (\omega_l)_j=1\}}  \delta_n^{r\beta}  \sum_{i} \phi^r\pr{\frac{i/n-b_j}{\delta_n}} 
 \overset{(a)}{=}  \sum_{\{j: (\omega_l)_j=1\}} \delta_n^{\beta r +1}n \left(\|\phi\|_r^r + o(1)\right) \\
 &\overset{(b)}{=} C_l m \delta_n^{\beta r +1}n \left(\|\phi\|_r^r + o(1)\right),
\end{split}
\end{align}

where to obtain equality (a) we have used the fact that for approximately $n \delta_n$ different values of $i$, $\frac{i/n - b_j}{\delta_n}$ falls in the support of $\phi$ and gives a non-zero value, and that when $n\de_n\to \infty$, there is an the approximation of the integral by Riemann sum with error $o(1)$.  To obtain (b) we have chosen the coefficient $C_l$  such that $C_l m$ is the number of nonzero elements in $\omega_l$. By construction, we know $ 1\geq C_l \geq 1/16$. 

Since $\frac{1}{1+\sigma_n^2}(1-\nu_l(i/n)) \le s_{l,i}:=\frac{1-\nu_l^2(i/n)}{{\sigma}^{2}_n +1}\le \frac{2}{1+\sigma_n^2}(1-\nu_l(i/n))$, 
we can use \eqref{eq:moments_f} to show that 
\begin{align}\label{eq:sr_first}
\frac{1}{(1+\sigma_n^2)^r} C_l m \delta_n^{\beta r +1}n \left(\|\phi\|_r^r + o(1)\right)<\sum s_{l,i}^{r}< \frac{2^r}{(1+\sigma_n^2)^r} C_l m \delta_n^{\beta r +1}n \left(\|\phi\|_r^r + o(1)\right).
\end{align}
Now, if we use the upper bounds for $\sum s_{l,i}^{2}, \sum s_{l,i}^{3}$, and $\sum s_{l,i}^{4}$ in \cref{eq:likelihood6}, we conclude that for showing \cref{eq:likelihood6} it is sufficient to prove that for some $p^*>0$ independent of $n,j$, we have

\begin{align}\label{eq:likelihood7}
\begin{split}
    & \mathbb{P}_{\nu_l}\left(\sum_{i} s_{l,i}\left(\frac{Y_{i}^{2}}{1+ {\sigma}^{2}_n}-1\right)- \frac{4}{(1+\sigma_n^2)^2} C_l m \delta_n^{2\beta  +1}n \left(\|\phi\|_2^2 + o(1)\right) \right.  \\
& \left. - \frac{8}{(1+\sigma_n^2)^3} C_l m \delta_n^{3\beta  +1}n \left(\|\phi\|_3^3 + o(1)\right) -2 \sqrt{3} \sqrt{\frac{16}{(1+\sigma_n^2)^4} C_l m \delta_n^{4\beta  +1}n \left(\|\phi\|_4^4 + o(1)\right) } >-2\lambda_{l} \log M \right)  \\
&>p^{*}+\frac{1}{3} .
\end{split}
\end{align}

 Now for our choice of $\delta_n$ in \cref{choice of delta}, we have 
$\frac{1}{4} \leq \frac{\delta_n^{2\beta + 1}n}{(1+\sigma_n^2)^2}\leq 1 $. Therefore, to prove \cref{eq:likelihood7}, it is sufficient to show that for some $p^*>0$ independent of $n,j$, we have

\begin{align}\label{eq:intermediate_step2}
\begin{split}
    & \mathbb{P}_{\nu_l}\left(\sum_{i} s_{l,i}\left(\frac{Y_{i}^{2}}{1+ {\sigma}^{2}_n}-1\right)- 4 C_l m  \left(\|\phi\|_2^2 + o(1)\right) \right.  \\
& \left. - 8 C_l m \delta_n^{\beta } \left(\|\phi\|_3^3 + o(1)\right) -2 \sqrt{3} \sqrt{16 C_l m \delta_n^{2\beta } \left(\|\phi\|_4^4 + o(1)\right) } >-2\lambda_{l} \log M \right) >p^{*}+\frac{1}{3}.
\end{split}
\end{align}

Note that we have assumed that $\de_n=o(1)$. Hence, proving \eqref{eq:intermediate_step2} is equivalent to proving that for some $p^*>$ we have 
\begin{align}\label{eq:likelihood8}
\mathbb{P}_{\nu_l}\left(\sum_{i} s_{l,i}\left(\frac{Y_{i}^{2}}{1+ {\sigma}^{2}_n}-1\right)- 4 C_l m \|\phi\|_2^2 +o(m) >-2\lambda_{l} \log M \right) >p^{*}+\frac{1}{3}.
\end{align}

As we discussed before, from \cref{GV} we know that $M=\left[2^{\frac{m}{8}}\right]$. Furthermore, from \eqref{eq:sr_first} we have
\begin{align}\label{eq:bd:sum:sli}
\sum s_{l,i}^{2}<\frac{4}{(1+\sigma_n^2)^2} C_l m \delta_n^{2\beta +1}n \left(\|\phi\|_2^2 + o(1)\right)
\leq 5 C_l m \|\phi\|_2^2, 
\end{align}
where to obtain the last inequality we used the fact that $\frac{\delta_n^{2\beta + 1}n}{(1+\sigma_n^2)^2}\leq 1$. 
Hence, for proving \cref{eq:likelihood8} it is sufficient to show that for some $p^*>0$ independent of $n,j$ we have
\begin{align}\label{eq:likelihood9}
\begin{split}
    &  \mathbb{P}_{\nu_l}\left(\sum_{i} \frac{s_{l,i}}{\sqrt{2 \sum s_{l,i}^{2} \left(1-s_{l,i}\right)^2}}\left(\frac{Y_{i}^{2}}{1+ {\sigma}^{2}_n}-1 + s_{l,i}\right) > \frac{m\left(9 C_l  \|\phi\|_2^2 -\lambda_l/5 \right)} {\sqrt{2 \sum s_{l,i}^{2} \left(1-s_{l,i}\right)^2}}\right) \\
& >p^{*}+\frac{1}{3}.
\end{split}
\end{align}

 As we discussed earlier, for large enough $n$, we have $0< \epsilon_{l,i}, \epsilon_{l,i}^{\prime} < s_{l,i} < 2 \nu_l(i/n)\leq 2 \delta_n^{\beta} \phi^{*}<\frac{1}{2}$. By
bounding $\sqrt{ \left(1-s_{l,i}\right)^2}$ on the right hand side of \cref{eq:likelihood8}, we realize that if we show the following it will imply \cref{eq:likelihood9}:
$$\begin{aligned}
&  \mathbb{P}_{\nu_l}\left(\sum_{i} \frac{s_{l,i}}{\sqrt{2 \sum s_{l,i}^{2} \left(1-s_{l,i}\right)^2}}\left(\frac{Y_{i}^{2}}{1+ {\sigma}^{2}_n}-1 + s_{l,i}\right) > \frac{m \sqrt{2}\left(9 C_l  \|\phi\|_2^2 -\lambda_l /5 \right)} {\sqrt{ \sum s_{l,i}^{2}}}\right)\\
& >p^{*}+\frac{1}{3}.
\end{aligned}
$$
Note that if we use \cref{BE}, proving the above equation will be equivalent to showing:

\begin{align}\label{LR:afterBerryEsseen}
& \mathbb{P}\left(\mathcal{N}(0,1)>\frac{m \sqrt{2}\left(9 C_l  \|\phi\|_2^2 -\lambda_l/5 \right)} {\sqrt{ \sum s_{l,i}^{2}}}\right)
>p^{*}+\frac{1}{3}.
\end{align}

Hence we have to answer two questions to finish the proof:
\begin{itemize}
\item Does \cref{LR:afterBerryEsseen} hold?

\item Do the conditions of \cref{BE} hold in deriving \cref{LR:afterBerryEsseen}?
\end{itemize}

We answer these two questions below:

\begin{itemize}
\item Justification of \cref{LR:afterBerryEsseen}:

Using the bounds we discussed before for $\sum s_{l,i}^{2}$, in order to show \cref{LR:afterBerryEsseen} it is sufficient to show that

$$
\begin{aligned}
& \mathbb{P}\left(\mathcal{N}(0,1)>\frac{m \sqrt{10}\left(9 C_l  \|\phi\|_2^2 -\lambda_l/5 \right)} {\sqrt{ C_l m \|\phi\|_2^2}}\right)
>p^{*}+\frac{1}{3} \Longleftrightarrow \\
& \mathbb{P}\left(\mathcal{N}(0,1)> \sqrt{10m}\left(9\|\phi\|_{2} \sqrt{C_{l}}-\frac{\lambda_{l}}{5 \sqrt{C_{l}}{ }\|\phi\|_{2}}\right) \right)
>p^{*}+\frac{1}{3}.
\end{aligned}
$$

Note that $\frac{1}{16} \leq C_{l} \leq 1$. Hence, if we choose $\lambda_{l}=\sqrt{\frac{C_{l}}{2}}$, and $\phi$ is small enough such that $\|\phi\|_{2}<1/10$, we have $ B_{l}<0$ where $B_{l}:=9\|\phi\|_{2} \sqrt{C_{l}}-\frac{\lambda_{l}}{5 \sqrt{C_{l}}{ }\|\phi\|_{2}} $. 
Hence, we need to show 
$$\mathbb{P}\left(\mathcal{N}(0,1)> \sqrt{10m} B_l \right)>p^{*}+\frac{1}{3}.$$
Since $B_l<0$, we can take $p^* = \frac{1}{6}$.

\item Verifying the conditions of \cref{BE}:
\end{itemize}

For $i=1, \ldots, n$, define $X_{i}:=s_{l,i}\left(\frac{Y_{i}^{2}}{1+ {\sigma}^{2}_n}-1 +s_{l,i}\right)$. It is straightforward to see that
$$
\begin{aligned}
& \mathbb{E}\left(X_{i}\right)=s_{l,i} \left(\frac{\mathbb{E}\left(Y_{i}^{2}\right)}{(1+ {\sigma}^{2}_n)}-1 +s_{l,i}\right)= s_{l,i}\left(\frac{{\sigma}^{2}_n +  \nu_{l}^{2}\left(i/n\right)}{1+ {\sigma}^{2}_n} -1 + s_{l,i}\right)\\
& = s_{l,i} (1 - s_{l,i} -1 + s_{l,i}) =0,
\end{aligned}$$
and
$$\begin{aligned}
&\mathbb{E}\left(X_{i}^{2}\right)=\operatorname{Var}\left(X_{i}\right)= \frac{s_{l,i}^{2}}{(1+ {\sigma}^{2}_n)^2} \operatorname{Var}\left(Y_{i}^{2}\right)=\frac{s_{l,i}^{2}}{(1+ {\sigma}^{2}_n)^2} \left(\mathbb{E}\left(Y_{i}^{4}\right)-\mathbb{E}\left(Y_{i}^{2}\right)^{2}\right) \\
& =s_{l,i}^{2} \cdot 2  \left( \frac{{\sigma}^{2}_n +  \nu_{l}^{2}\left(i/n\right)}{1+ {\sigma}^{2}_n}\right)^2 =2\left(1-s_{l,i}\right)^{2} s_{l,i}^{2}.
\end{aligned}
$$

Moreover,
$$\begin{aligned}
& \mathbb{E}\left|X_i\right|^3 \leq s_{l, i}^3 \mathbb{E}\left[\left(\frac{Y_i^2}{1+{\sigma}^{2}_n}-\left(1-s_{l,i}\right)\right)^3\right] =s_{l,i}^3 \mathbb{E}\left[\left(\frac{Y_i^2-\left({\sigma}^{2}_n +\nu_l^2\left(i/n\right)\right)}{1+{\sigma}^{2}_n}\right)^3\right] \\
& \leq\left(\frac{s_{l, i}}{1+{\sigma}^{2}_n}\right)^3\left[\mathbb{E} Y_i^6+3 \mathbb{E} Y_i^4\left({\sigma}^{2}_n +\nu_l^2(i/n)\right)+3 \mathbb{E} Y_i^2\left({\sigma}^{2}_n +\nu_l^2(i / n)\right)^2\right. \\
& \left.+\left({\sigma}^{2}_n+\nu_l^2(i/ n)\right)^3\right] \\
& =\left(\frac{s_{l, i}}{1+{\sigma}^{2}_n}\right)^3\left(15\left({\sigma}^{2}_n +\nu_l^2(i / n)\right)^3+9\left({\sigma}^{2}_n+\nu_l^2(i / n)\right)^3\right. \\
& \left.  +3\left({\sigma}^{2}_n+\nu_l^2(i/n)\right)^3+\left({\sigma}^{2}_n +\nu_l^2(i/n)\right)^3\right)=28 s^3 _{l,i}\left(1-s_{l,i}\right)^3.
\end{aligned}
$$
Now we are ready to check the conditions of \cref{BE}. Note that we have
\begin{align}
\begin{split}
    \frac{\mathbb{E} (|X_i|^3)}{\mathbb{E} (X_i^2)} &= \frac{28 s_{l,i}^3 (1- s_{l,i}^3)}{2 (1- s_{l,i}^2) s_{l,i}^2} = 14 s_{l,i} (1- s_{l,i}) = 14 \frac{1-\nu_l^2(i / n) }{1+ \sigma_n^2} \frac{\sigma_n^2+\nu_l^2(i / n)}{1+ \sigma_n^2}  \\ 
& \leq 14 \frac{1-\nu_l^2(i / n) }{1+ \sigma_n^2} \leq 28 \frac{1 -\nu_l(i / n)}{1+ \sigma_n^2} \leq 28 (1 -\nu_l(i / n)) \leq 28 \delta_n^{\beta} \phi^*.
\end{split} 
\end{align}
Hence, we set $\kappa$ (defined in the statement of \cref{BE}) in the following way: 
$$\kappa=  28 \delta_n^{\beta} \phi^*.$$ Leta $S_{*}:=\frac{\sum_i X_{i}}{\sqrt{\sum_i \mathbb{E} X_{i}^{2}}},$ and $F_{S_*}$ denote the CDF of random variable $X_*$. From \cref{BE} we have 
\begin{align}\label{eq:CLT:upper}
\begin{split}
    & \sup _{x}\left|F_{S_{*}}(x)-\Phi(x)\right|<C_{o}\left(\frac{28 \delta_n^{\beta} \phi^* }{\sqrt{2 \sum s_{l,i}^{2} \left(1-s_{l,i}\right)^2}}\right) <\frac{C_{0}^{\prime} \delta_n^{\beta} }{\left(\sqrt{ \sum s_{l,i}^{2}}\right)} \overset{(a)}{=} O\left(\frac{ \delta_n^{\beta} }{\sqrt{m} }\right)  \\
& \overset{(b)}{=} O\left(\delta_n^{\beta+1/2}\right)= o(1),
\end{split}
\end{align}
where to obtain Equality (a) we used the bound we obtained for $\sum s_{l,i}^{2}$ in \eqref{eq:bd:sum:sli} and to obtain equality $b$ we used the fact that $m = [\delta_n^{-1}]$. Hence, \eqref{eq:CLT:upper} implies that 
\begin{align*}
\Big|\mathbb{P}\left(S_{*}> x \right)-\mathbb{P}(\mathcal{N}(0,1)> x )\Big|= o(1),
\end{align*}
and therefore
$$
\begin{aligned}
& \left|\mathbb{P}_{\nu_l}\left(\sum_{i}\left(\frac{Y_{i}^{2}}{1 + {\sigma}^{2}_n}-1+s_{l,i}\right)\frac{s_{i}}{\sqrt{2\sum_i s^2_i \left(1-s_{l,i}\right)^2}}>x\right)-\mathbb{P}\left(\mathcal{N}(0,1)> x \right)\right|\\
&=o(1). \end{aligned}.$$ 

\end{proof}

\subsection{Proof of the upper bound of \cref{l_2}}\label{ssec:upper:lpe}

In this subsection, to obtain an upper bound, we evaluate the performance of the modified version of the local polynomial estimator (LPE) approach presented in Section 1.6 of \cite{Nonp}.  

When $\sigma_n^4=\Omega(n)$, we can take the estimator $\hfn=0$ and it follows that
\begin{align}
R_2\pr{\Sigma_{\mathfrak{h}}(\beta,L),\sigma_n}\le \sup_{f\in \Sigma_{\mathfrak{h}}(\beta,L)}\|f\|_2^2\le 1
\end{align}

To an obtain an estimate of $f\in \Sigma_{\mathfrak{h}}(\beta,L)$ in the case when $\sigma_n^4=o(n)$, we  adopt the kernel method and LPE . Let $K:\R \to \R$ be a bounded, non-negative compactly supported kernel function, satisfying $\int_{\R}K(u)du=1$ and $K(u)>0$ for $u\in [-1,1]$. To simplify our calculations we consider the following kernel $K(\cdot)$:
\begin{equation}\label{the kernel function}
    K(u):=\frac{1}{2}I(|u|\le 1).
\end{equation}
Define
\begin{align}
    \mathbf{U}(u):=& \pr{1,u,{u^2\over 2!},...,{u^k \over k!}}^T.
\end{align}
For a given bandwidth $h_n>0$ satisfying
\begin{equation*}
    nh_n\to \infty, \text{ as }n\to \infty,
\end{equation*}
we define the function 
\begin{equation}\label{eq:W*def}
    W_{ni}^*(x):=\frac{1}{nh_n}\mathbf{U}^T(0)\mathcal{B}_{nx}^{-1}\mathbf{U}\pr{\frac{i/n-x}{h_n}}K\pr{\frac{i/n-x}{h_n}} 
\end{equation}
where the matrix $\mathcal{B}_{nx}$ is defined as 
\begin{equation}
    \mathcal{B}_{nx}=\sum_{i=1}^n \mathbf{U}\pr{\frac{i/n-x}{h_n}}\mathbf{U}^T\pr{\frac{i/n-x}{h_n}}K\pr{\frac{i/n-x}{h_n}}
\end{equation}

Defining $g(x):=f^2(x)$ we obtain an estimate of $g$ using the following: 
\begin{equation}
\hat{g}_n(x):=\sum_{i=1}^n \pr{Y_i^2-\sigma_n^2} \cdot W_{ni}^*(x),
\end{equation}
where
\begin{equation}
Y_i=f(i/n)\xi_i+\tau_i.
\end{equation}

Our goal is to find an upper bound for the mean square error. 
\begin{align}
\begin{split}
         & \E_f \br{\uint\br{\hat{g}_n(x)-f^2(x)}^2 dx} \\
     = & \uint b^2(x)dx+\uint \sigma^2(x)dx,
\end{split}
\end{align}
where 
\begin{align}
    b(x):=\E_f\br{\hat{g}_n(x)}-f^2(x)
\end{align}
represents the bias of $\hat{g}_n(x)$ and
\begin{align}
\sigma^2(x):=\E_f\br{\pr{\hat{g}_n(x)-\E_f\br{\hat{g}_n(x)}}^2}
\end{align}
denotes the variance of the estimator. Our aim is to uniformly bound $b(x)$ and $\sigma(x)$ for any $x\in [0,1]$. Here we recall a few properties of $W_{ni}^*$ and $\mathcal{B}_{nx}$ as follows:

\begin{proposition}[Proposition 1.12 and Lemma 1.3, \cite{Nonp}]\label{basic properties of W}
Under our choice of kernel function as in \cref{the kernel function}, there exists a constant $C_*$ independent of $n$ such that
    \begin{enumerate}
    \item $\sum_{i=1}^n W_{ni}^*(x)=1$ and $\sum_{i=1}^k(i/n-x)^k W_{ni}^*(x)=0$, where $k$ is the largest integer stricly smaller than $\beta$.
        \item $\sup_{i,x}\abs{W_{ni}^*(x)}\le \frac{C_*}{nh_n}$;
        \item $\sum_{i=1}^n \abs{W_{ni}^*(x)}\le C_*$
        \item $\mathcal{B}_{nx}$ is non-singular. Moreover,
        \begin{equation} \nm{\mathcal{B}_{nx}^{-1}v}_2\le C_*\nm{v}_2
        \end{equation}
    \end{enumerate}
\end{proposition}

Now we shall estimate the upper bound by estimating $b(x)$ and $\sigma^2(x)$ separately. Note that
\begin{align}
     b(x)=\E_f\br{\hat{g}_n(x)}-f^2(x) =\sum_{i=1}^n f^2(i/n)W_{ni}^*(x)-f^2(x)
\end{align}

In view of \cref{lem: product is holder}, $f^2 \in \Sigma(\beta, C_{\beta,L})$ for some constant $C_{\beta,L}>0$. Hence, following the same steps as the ones in the proof of Theorem 1.6.1 and Theorem 1.6.5 of \cite{Mainbook} we conclude that 
\begin{align}
    &\abs{\sum_{i=1}^n f^2(i/n) W_{ni}^*(x)-f^2(x)}
   = O(h_n^{\beta}).
\end{align}

We should emphasize that the constants are independent of $x\in [0,1]$, and hence the following upper bound
\begin{equation}
b^2(x)=O\pr{h_n^{2\beta}}
\end{equation} 
holds uniformly on $x$. Our next step to find an upper bound for $\sigma^2(x)$. We have

\begin{align}
 \begin{split}
        \sigma^2(x)
    = & \E_f\br{\pr{\hgn(x)-\E_f\br{\hgn(x)}}^2}  \\
    = & \E_f\br{\pr{\sum_{i=1}^n \pr{Y_i^2-\sigma_n^2} W_{ni}^*(x)-\sum_{i=1}^n f^2(i/n)W_{ni}^*(x)}^2}  \\
   =& \E_f\br{\pr{\sum_{i=1}^n \br{f^2(i/n)(\xi_i^2-1)+(\tau_i^2-\sigma_n^2)+2f(i/n)\xi_i\tau_i}W_{ni}^*(x)}^2} \\
   \overset{(a)}{=}& \E\br{\sum_{i=1}^n \br{f^4(i/n)(\xi_i^2-1)^2+(\tau_i^2-\sigma_n^2)^2+4f(i/n)^2\xi_i^2\tau_i^2}W_{ni}^*(x)^2} \\
   \overset{(b)}{\le} & \E\br{\sum_{i=1}^n \br{(\xi_i^2-1)^2+(\tau_i^2-\sigma_n^2)^2+4\xi_i^2\tau_i^2}W_{ni}^*(x)^2} \\
   \overset{(c)}{=}& O(\max(1,\sigma_n^4)) \sum_{i=1}^n W_{ni}^*(x)^2 \\
    = & O(\max(1,\sigma_n^4)) O\pr{\frac{1}{nh_n}},
 \end{split}
\end{align}

where $(a)$ follows from our assumptions that $\xi_i$'s and $\tau_j$'s are independent so the cross product terms have zero mean and $(b)$ uses our assumption that $|f|\le 1$. $(c)$ follows from $\E[(\xi_i-1)^2]=O(1)$, $\E(\tau_i^2-\sigma_n^2)^2=O(\sigma_n^4)$ and $\E(\xi_i^2\tau_i^2)=O(\sigma_n^2)$.

To summarize, we have
\begin{align}\label{eq:upper_biasPvar}
  \begin{split}
         &\E_f \br{\uint\br{\hat{f}_n(x)-f(x)}^2 dx}  \\
   = & \uint b^2(x)dx+\uint \sigma^2(x)dx  \\
   \le & O\pr{h_n^{2\beta} + \frac{\max\pr{1,\sigma_n^4}}{nh_n}}. 
  \end{split}
\end{align}

Recall that we have the restrictions that $h_n\to 0$ and $nh_n \to \infty$ as $n\to \infty$. Optimizing $h_n^{2\beta} + \frac{\max\pr{1,\sigma_n^4}}{nh_n}$ over $ h_n>0$ yields 
\begin{equation}\label{optimized choice for $h_n$}
     \widetilde h_n\sim \left(\frac{\max\pr{1,\sigma_n^4}}{n}\right)^{1 \over 2\beta+1}.
 \end{equation}
For this choice of $h_n$, we always have $nh_n\to \infty$. In this case
\begin{equation*}
     \widetilde h_n^{2\beta}\sim \left(\frac{\max\pr{1,\sigma_n^4}}{n}\right)^{2\beta \over 2\beta+1}.
 \end{equation*}

If we plug in this choice of $\tilde{h}_n$ with \eqref{eq:upper_biasPvar} we obtain 
    \begin{equation*}
\E_f \br{\uint\br{\hat{f}_n(x)-f(x)}^2 dx}  = O\pr{\left(\frac{\max\pr{1,\sigma_n^4}}{n}\right)^{2\beta \over 2\beta+1}}.
    \end{equation*}

We can now finish the proof with the following simple steps: 
\begin{align}\label{passing from square to no square}
\begin{split}
    R_2\pr{\Sigma_{\mathfrak{h}}(\beta,L)}
    =&\inf_{\hfn}\sup_{f\in \Sigma_{\mathfrak{h}}(\beta,L)}\E_f\nm{\hfn-f}_2^2  \\
    \le &\inf_{\hfn\ge 0}\sup_{f\in \Sigma_{\mathfrak{h}}(\beta,L)}\E_f\nm{\hfn-f}_2^2 \\
    \overset{(a)}{\le} &\mathfrak{h}^2 \inf_{\hgn\ge 0}\sup_{f\in \Sigma_{\mathfrak{h}}(\beta,L)}\E_f\nm{\hgn-f^2}_2^2 \\
    \le &\mathfrak{h}^2 \inf_{\hgn}\sup_{f\in \Sigma_{\mathfrak{h}}(\beta,L)}\E_f\nm{\hgn-f^2}_2^2 = O_{\beta,L,\mathfrak{h}}\pr{\left(\frac{\max\pr{1,\sigma_n^4}}{n}\right)^{2\beta \over 2\beta+1}}.
\end{split}
\end{align}
where for $(a)$ we have used that assumption that $f\ge \mathfrak{h}$.

\subsection{Proof of the lower bound of \cref{l_infty}}
%The proof of this theorem follows a similar approach to the one used in the proof of \cref{l_2}. 
%Hence, we do not repeat our main steps. 

%\subsubsection{Roadmap of the proof }

\quad
 Our proof strategy for the lower bound mirrors that of \cref{l_2}. As a result, we will omit some steps and focus on explaining the construction of the set \(\{\nu_0, \nu_1, \ldots, \nu_m\}\). Note that our construction is simpler than the one in \cref{roadmap:lb}. For $\sigma_n^4= O\pr{n^{1-\eta}}$ for some $\eta\in (0,1)$, let 
\begin{align}\label{choice of delta for infinity norm}
\delta_n=\left(\log \frac{n}{(1+\sigma_n^2)^2}  \middle/ \frac{n}{(1+\sigma_n^2)^2}\right)^{\frac{1}{2\beta+1}}.
\end{align}

Define $\nu_{0} := 1$, and for $j=1, 2, \ldots, m ,\, m=[\delta_n^{-1}]$, define  $\nu_{j}:=f_{j,n}$, where 

$$
\begin{aligned}
& f_{j,n}(x):=1-\delta_n^{\beta} \phi\left(\frac{x-b_j}{\delta_n}\right), \quad b_{j}=(2 j-1) \frac{\delta_n}{2}.  
\end{aligned}
$$
Similar to \cref{belonging}, we can prove that $f_{j,n} \in \Sigma(\beta,L)$ for $j=1, \ldots, m$. Also, if $\phi^*$ denotes the maximum of the function $\phi$, then we will have
 \begin{align}
 \|\nu_{j}- \nu_{k}\|_{\infty} = \phi^{*} \delta_n^{\beta}, 
 \end{align}
 where $\phi^{*} :=\lVert \phi \rVert_{\infty}:= \sup_{x} |\phi(x)|.$
 Our next goal is to use \cref{thm1} with $d$ being the $L_{\infty}$ distance. Towards this goal, first we need to calculate the likelihood ratio, which is 
 \begin{align*}
& \Lambda\left(\nu_{0}, \nu_{j}\right)=\prod_{i} \sqrt{\frac{{\sigma}^{2}_n +  \nu_{j}^{2}\left(i/n\right)}{1+ {\sigma}^{2}_n}} \exp \left[\sum_i \frac{Y_{i}^{2}}{2 ({\sigma}^{2}_n + \nu_{j}^{2}\left(i/n\right))}-\sum \frac{Y_{i}^{2}}{2(1+{\sigma}^{2}_n)}\right] \\
& =\exp \left[\sum_{i} \frac{Y_{i}^{2}}{2}\left( \frac{1}{ {\sigma}^{2}_n + \nu_{j}^{2}\left(i/n\right)}- \frac{1}{1+ {\sigma}^{2}_n}\right)+\frac{1}{2}\sum_{i} \log \frac{\sigma_n^2+\nu_{j}^{2}\left(i/n\right)}{\sigma_n^2+1}\right]\\
& =\exp \left[\sum_{i} \frac{Y_{i}^{2}}{2(1 + {\sigma}^{2}_n)}\left( \frac{\sigma_n^2+1}{\sigma_n^2+\nu_{j}^{2}\left(i/n\right)}- 1\right)+\frac{1}{2}\sum_{i} \log \frac{\sigma_n^2+\nu_{j}^{2}\left(i/n\right)}{\sigma_n^2+1}\right].
\end{align*}

Now to guarantee the second condition of \cref{thm1}, we need to prove the following lemma:

\begin{lemma}\label{lem:likelihoodratio2}
Let $\nu_j,j=0,1,...,m$ be defined as above, then there exists some $\lambda< 1$, such that for all $j=1,\ldots,m$,
$$\mathbb{P}_{\nu_{j}}\left(\Lambda\left(\nu_{0}, \nu_{j}\right)> m^{-\lambda_j}\right)>p^{*},$$
whenever $\lambda_j< \lambda$.
\end{lemma}

The proof will be given in \cref{pofll2}. Combining the results of \cref{thm1} (with $M=m$) and \cref{lem:likelihoodratio2} gives us

$$
\begin{aligned}
& \max_{0\leq j\leq m}\mathbb{P}_{\nu_j}\left(\|\hat{\nu}_n- \nu_j\|_\infty\geq \frac{ \delta_n^{\beta}\phi^*}{2} \right)\geq 1/4.\end{aligned}$$
Using Markov's inequality, for any estimator $\hat{\nu}_n$,
\begin{align*}
& \max_{0\leq j\leq m}\mathbb{E}_{\nu_j}\left(\|\hat{\nu}_n- \nu_j\|_\infty \right)\geq \frac{ \delta_n^{\beta}\phi^*}{8}.
\end{align*}
Hence, in the case $\sigma_n^4=O(n^{1-\eta})$ for some $\eta\in (0,1)$, we have
\begin{align*}
R_\infty(\Sigma_{\mathfrak{h}}(\beta, L), {\sigma}_n)\geq \frac{ \delta_n^{\beta}\phi^*}{8} = &\Theta\left(\left(\log \frac{n}{(1+\sigma_n^2)^2}  \middle/ \frac{n}{(1+\sigma_n^2)^2}\right)^{\frac{\beta}{2\beta +1}}\right)\\
=&\Theta_{\eta} \left(\pr{\max(1,\sigma_n^4)\frac{\log n}{n}}^{\frac{\beta}{2\beta +1}}\right)
.\end{align*}
\subsubsection{Proof of \cref{lem:likelihoodratio2}}\label{pofll2} We need to show that there exists $p^*>0$ independent of $j, n$ and $\lambda\in (0,1)$ such that for $j=1, \ldots,m $,

\begin{align}\label{eq:inf:l:1}
\begin{split}
    & \mathbb{P}_{\nu_j}\left(\sum_{i} \frac{Y_{i}^{2}}{1 + {\sigma}^{2}_n}\left( \frac{\sigma_n^2+1}{\sigma_n^2+f_{j,n}^{2}\left(i/n\right)}- 1\right)+\sum_{i} \log \frac{\sigma_n^2+f_{j,n}^{2}\left(i/n\right)}{\sigma_n^2+1}>-2\lambda_{j} \log m\right) \\
& >p^{*}.
\end{split}  \end{align}

If we put $s_{j,i}:=\frac{1-f_{j.n}^2(i/n)}{\sigma_n^2+1},$ then \cref{eq:inf:l:1} is equivalent to 

\begin{align}\label{eq:linf:l:3}
& \mathbb{P}_{\nu_j} \left[\sum_{i} \frac{Y_{i}^{2}}{1+ {\sigma}^{2}_n}\left( \frac{1}{1 - s_{j,i}}- 1\right)+\sum_{i} \log (1- s_{j,i})>-2\lambda_{j} \log m \right]>p^{*}.
\end{align}
Note that according to Taylor expansion, we know that for some $\epsilon_{j,i}, \epsilon_{j,i}^{\prime} \in\left(0, s_{j,i}\right),$ we have
$$
\begin{aligned}
& \frac{1}{1-s_{j,i}}=1+s_{j,i}+\frac{s_{j,i}^{2}}{\left(1-\epsilon_{j,i}\right)^{3}}\\
& \log (1-s_{j,i})=-s_{j,i}-s_{j,i}^{2} \frac{1}{2\left(1-\epsilon_{j,i}^{\prime}\right)^{2}}.
\end{aligned}$$
Using these two equations, we can simplify the problem of showing \cref{eq:linf:l:3} to the problem of proving:

\begin{align}\label{eq:3.24}
\begin{split}
    & \mathbb{P}_{\nu_j}\left[\sum_{i} \frac{Y_{i}^{2}}{1+ {\sigma}^{2}_n}\left(s_{j,i}+\frac{ s_{j,{i}}^{2}}{\left(1-\epsilon_{j,i}\right)^{3}}\right)-\sum_{i}\left(s_{j,i}+\frac{s_{j,i}^{2}}{2\left(1-\epsilon_{j,i}^{\prime}\right)^{2}}\right) >-2\lambda_{j} \log m\right] >p^{*}. 
\end{split}
\end{align}
For large enough $n$, we have
\begin{align}\label{trivial bound for s}
    0< \epsilon_{j,i}, \epsilon_{j,i}^{\prime} < s_{j,i} \leq 2 \delta_n^{\beta} \phi^{*}<\frac{1}{2}.
\end{align}
 Hence to prove \cref{eq:3.24} it is sufficient to prove
$$
\begin{aligned}
& \mathbb{P}_{\nu_j}\left[\sum_{i} \frac{Y_{i}^{2}}{1+ {\sigma}^{2}_n}\left(s_{j,i}+ s_{j,i}^{2}\right)-\sum_{i}\left(s_{j,i}+2 s_{j,i}^{2}\right)>-2\lambda_{j} \log m\right]>p^{*}  
\end{aligned},
$$
which is equivalent to 
\begin{align}\label{3.25}
\mathbb{P}_{\nu_j}\left[\sum_{i} s_{j,i}\left(\frac{Y_{i}^{2}}{1+ {\sigma}^{2}_n}-1\right)+\sum_i s_{j,i}^{2} \left( \frac{Y_{i}^{2}}{1+ {\sigma}^{2}_n}-2\right)>-2\lambda_{j} \log m\right]>p^{*} .
\end{align}

Using \cref{lM}, we have with the choice $t =\log 3$,

\begin{align}\label{3.26}
\begin{split}
    &  \mathbb{P}_{\nu_j}\left(\sum_i Y_{i}^{2}\left(\frac{1}{{\sigma}^{2}_n+f_{j,n}^2{(i / n)}}\right)  s_{j,i}^{2} \frac{\sigma_n^2+ f_{j,n}^{2}(i/n)}{\sigma_n^2+1} \leq \sum_{i} s_{j,i}^{2} \frac{\sigma_n^2+f_{j,n}^{2}(i/n)}{\sigma_n^2+1} \right.\\
& \left. -2 \sqrt{\log 3}  \sqrt{\sum_{i}  s_{j,i}^{4} \left(\frac{\sigma_n^2+ f_{j,n}^{2}(i/n)}{\sigma_n^2+1 }\right)^2}\right)\leq \frac{1}{3}.
\end{split}
\end{align}

Recalling the definition of $s_{j,i}$, we know that $1-s_{j,i}=\frac{\sigma_n^2+f_{j,n}^{2}(i/n)}{\sigma_n^2+1}$. Hence \cref{3.26} is equivalent to 
$$
\begin{aligned}
&  \mathbb{P}_{\nu_j}\left(\sum_i Y_{i}^{2}\left(\frac{1}{{\sigma}^{2}_n +f_{j,n}^2{(i / n)}}\right)  s_{j,i}^{2} (1-s_{j,i}) \leq \sum_{i} s_{j,i}^{2} (1-s_{j,i}) \right.\\
& \left. -2 \sqrt{\log 3} \sqrt{\sum_{i}  s_{j,i}^{4} (1-s_{j,i})^2}\right)\leq \frac{1}{3}.
\end{aligned}
$$

Therefore to show \cref{3.25}, we only need to show for some $p^*\geq0$,

\begin{align}\label{3.27}
\begin{split}
    & \mathbb{P}_{\nu_j}\left(\sum_{i} s_{j,i}\left(\frac{Y_{i}^{2}}{1+ {\sigma}^{2}_n }-1\right)- \sum_{i} s_{j,i}^{2} (1+s_{j,i})-2 \sqrt{\log 3} \sqrt{\sum_{i}  s_{j,i}^{4} } >-2\lambda_{j} \log m \right) \\
&>p^{*}+\frac{1}{3} . 
\end{split}
\end{align}
Note that $\forall r \in \mathbb{Z}^{+}$,
$$
\begin{aligned}
&\sum_{i} (1-f_{j,n}(i/n))^{r}=\delta_n^{r \beta} \sum_{i} \phi^{r} \left( \frac{i/n-b_j}{\delta_n}\right) =\delta_n^{r \beta+1} n \sum_{i} \frac{1}{n \delta_n} \phi^{r} \left(\frac{i/n-b_j}{\delta_n}\right)\\
&=\delta_n^{r \beta+1} n\left(\|\phi\|_{r}^{r}+o(1)\right).
\end{aligned}
$$

Moreover, $\frac{1}{1
+\sigma_n^2}\pr{1-f_{j,n}(i/n)}<s_{j,i} < \frac{2}{1+\sigma_n^2}\pr{1-f_{j,n}(i/n)}$, and therefore we have
\begin{align}\label{b}
{1\over (1+\sigma_n^2)^r}  \delta_n^{r\beta  +1}n \left(\|\phi\|_r^r + o(1)\right)<\sum_i s_{j,i}^{r}< {2^r\over (1+\sigma_n^2)^r}  \delta_n^{r\beta  +1}n \left(\|\phi\|_r^r + o(1)\right).    
\end{align}

If we use the upper bounds for $\sum s_{j,i}^{2}, \sum s_{j,i}^{3}$, and $\sum s_{j,i}^{4}$, to guarantee \cref{3.27} it is enough to show 

$$
\begin{aligned}
& \mathbb{P}_{\nu_j}\left(\sum_{i} s_{j,i}\left(\frac{Y_{i}^{2}}{1+ {\sigma}^{2}_n}-1\right)- {4\over (1+\sigma_n^2)^2}   \delta_n^{2\beta  +1}n \left(\|\phi\|_2^2 + o(1)\right) \right. \\
& \left. - {8\over (1+\sigma_n^2)^3}  \delta_n^{3\beta  +1}n \left(\|\phi\|_3^3 + o(1)\right) -2 \sqrt{\log 3} \sqrt{{16\over (1+\sigma_n^2)^4} \delta_n^{4\beta  +1}n \left(\|\phi\|_4^4 + o(1)\right) } >-2\lambda_{j} \log m \right) \\
&>p^{*}+\frac{1}{3} . 
\end{aligned}
$$
Now our choice of $\delta_n$ gives ${\delta_n^{2\beta + 1}n \over (1+\sigma_n^2)^2}= \log \frac{n}{(1+\sigma_n^2)^2}$ and substitute it into the above inequality, it suffices to establish
\begin{align}\label{before merging}
    \begin{split}
        & \mathbb{P}_{\nu_j}\left(\sum_{i} s_{j,i}\left(\frac{Y_{i}^{2}}{1+ {\sigma}^{2}_n}-1\right)- 4 \log \frac{n}{(1+\sigma_n^2)^2} \left(\|\phi\|_2^2 + o(1)\right) \right. \\
& \left. - 8  \delta_n^{\beta }\log \frac{n}{(1+\sigma_n^2)^2} \left(\|\phi\|_3^3 + o(1)\right) -2 \sqrt{\log 3} \sqrt{16 \delta_n^{2\beta }\log \frac{n}{(1+\sigma_n^2)^2} \left(\|\phi\|_4^4 + o(1)\right) } >-2\lambda_j \log m\right) \\
& >p^{*}+\frac{1}{3}.
    \end{split}
\end{align}
Rearranging the terms and noticing that $\delta_n=\left(\log \frac{n}{(1+\sigma_n^2)^2}  \middle/ \frac{n}{(1+\sigma_n^2)^2}\right)^{\frac{1}{2\beta+1}}$ is bounded and that
\begin{equation}
    \log m= \log \br{\de_n^{-1}}= \frac{1}{2\beta+1}\pr{\log \frac{n}{(1+\sigma_n^2)^2}-\log\log \frac{n}{(1+\sigma_n^2)^2}},
\end{equation}
and by merging all lower order terms involving $o\pr{\log \frac{n}{(1+\sigma_n^2)^2}}$, it suffices to show

\begin{align}\label{3.28}
\begin{split}
    & \mathbb{P}_{\nu_j}\left(\sum_{i} s_{j,i}\left(\frac{Y_{i}^{2}}{1+ {\sigma}^{2}_n}-1\right)- 4 \log \frac{n}{(1+\sigma_n^2)^2} \|\phi\|_2^2 +o\pr{\log \frac{n}{(1+\sigma_n^2)^2}} >-\frac{2\lambda_{j}}{2\beta+1} \log \frac{n}{(1+\sigma_n^2)^2} \right)\\
& >p^{*}+\frac{1}{3}.
\end{split}
\end{align}
Equivalently, we need to show that
$$
\begin{aligned}
& \mathbb{P}_{\nu_j}\left(\sum_{i} s_{j,i}\left(\frac{Y_{i}^{2}}{1+ {\sigma}^{2}_n}-1\right) > 4 \log \frac{n}{(1+\sigma_n^2)^2}  \|\phi\|_2^2 +o\pr{\log \frac{n}{(1+\sigma_n^2)^2}}-\frac{2\lambda_{j}}{2\beta+1} \log \frac{n}{(1+\sigma_n^2)^2} \right) \\
& >p^*+\frac{1}{3}.\end{aligned}$$
Since
\begin{align}\label{r=2}
\log \frac{n}{(1+\sigma_n^2)^2} \left(\|\phi\|_r^r + o(1)\right)<\sum_i s_{j,i}^{2}< 4 \log \frac{n}{(1+\sigma_n^2)^2} \left(\|\phi\|_r^r + o(1)\right),    
\end{align} 
it is sufficient to show that
\begin{align}\label{3.29}
\begin{split}
    &  \mathbb{P}_{\nu_j}\left(\sum_{i} \frac{s_{j,i}}{\sqrt{2 \sum_i s_{j,i}^{2} \left(1-s_{j,i}\right)^2}}\left(\frac{Y_{i}^{2}}{1+ {\sigma}^{2}_n}-1 + s_{j,i}\right)  \right. \\
& \left.  > \frac{\log \frac{n}{(1+\sigma_n^2)^2} \left(9  \|\phi\|_2^2 -\frac{2\lambda_j}{2\beta + 1} + o(1) \right)} {\sqrt{2 \sum_i s_{j,i}^{2} \left(1-s_{j,i}\right)^2}}\right) >p^*+\frac{1}{3}.
\end{split}
\end{align}
 From \cref{trivial bound for s} we have $0<s_{j,i}<\frac{1}{2}$, therefore it suffices to show
\begin{align}\label{3.30}
\begin{split}
    &  \mathbb{P}_{\nu_j}\left(\sum_{i} \frac{s_{j,i}}{\sqrt{2 \sum s_{j,i}^{2} \left(1-s_{j,i}\right)^2}}\left(\frac{Y_{i}^{2}}{1+ {\sigma}^{2}_n}-1 + s_{j,i}\right) \right. \\
& \left. > \frac{\sqrt{2}\log \frac{n}{(1+\sigma_n^2)^2}  \left(9   \|\phi\|_2^2 -\frac{2\lambda_j}{2\beta + 1} + o(1)\right)} {\sqrt{ \sum s_{j,i}^{2}}}\right) >p^*+\frac{1}{3}.
\end{split}
\end{align}

Now by \cref{BE}, the LHS in \cref{3.30} is almost a standard Gaussian random variable. Hence, it is enough to show

$$
\begin{aligned}
& \mathbb{P}\left(\mathcal{N}(0,1)> \frac{\sqrt{2} \log \frac{n}{(1+\sigma_n^2)^2}  \left(9   \|\phi\|_2^2 -\frac{2\lambda_j}{2\beta + 1} + o(1)\right)} {\sqrt{ \sum s_{j,i}^{2}}}\right)
>p^*+\frac{1}{3}.
\end{aligned}
$$
Using \cref{b} and bounding $\sum s_{j,i}^{2}$, the above inequality is guaranteed by showing

$$
\begin{aligned}
& \mathbb{P}\left(\mathcal{N}(0,1)>\frac{\sqrt{10}\log \frac{n}{(1+\sigma_n^2)^2}  \left(9   \|\phi\|_2^2 -\frac{2\lambda_j}{2\beta + 1}+ o(1)\right)} {\sqrt{\log \frac{n}{(1+\sigma_n^2)^2}} \|\phi\|_2} \right)
>p^*+\frac{1}{3},
\end{aligned}
$$

which is equivalent to 
\begin{align}\label{f}
 \mathbb{P}\left(\mathcal{N}(0,1)> \sqrt{10\log \frac{n}{(1+\sigma_n^2)^2} }\left(9\|\phi\|_{2} - \frac{2\lambda_j}{\|\phi\|_2(2\beta + 1)}+ o(1)\right) \right)
>p^*+\frac{1}{3}.   
\end{align}

Now if $\lambda_{j}:=\frac{1}{2}$ and $\phi$ is such that $\|\phi\|_{2}<\frac{1}{4\sqrt{(2 \beta+1)}}$ (Notice that we can always make $\phi$ small enough by scaling.), then we have $$9\|\phi\|_{2} - \frac{2\lambda_j}{\|\phi\|_2(2\beta + 1)}+ o(1)<0$$ and \cref{f} holds with $p^{*}=\frac{1}{6}$. 
%So $\left(\frac{\log n}{n}\right)^{\frac{1}{2 \beta +1}}$ is our lower bound.
Justification of the step that we have used \cref{BE} in the above argument is the same as \cref{proofl_2} and therefore skipped.

\subsection{Proof of the upper bound of \cref{l_infty}}
As in \cref{ssec:upper:lpe}, we consider the estimator
\begin{equation}
\hat{g}_n(x):=\sum_{i=1}^n \pr{Y_i^2-\sigma_n^2} \cdot W_{ni}^*(x),
\end{equation}
where
\begin{equation}
Y_i=f(i/n)\xi_i+\tau_i.
\end{equation}
From the bias-variance decomposition we obtain: 
\begin{align*}
\begin{split}
       &\E \br{\nm{\hgn-f^2}_{\infty}}\\
\le&\E \br{\nm{\E \hgn-f^2}_{\infty}+\nm{\hgn-\E \hgn }_{\infty}}\\
=& \nm{\E \hgn-f^2}_{\infty}+\E\nm{\hgn-\E \hgn }_{\infty}
\end{split}
\end{align*}

For the first part, as in \cref{ssec:upper:lpe}, we have uniformly in $x\in [0,1]$,
\begin{align}\label{bias part infinity norm}
     \begin{split}     &\abs{\E\br{\hat{g}_n(x)}-f^2(x)}\\
    =&\abs{\sum_{i=1}^n f^2(i/n)W_{ni}^*(x)-f^2(x)}\\
    =& O(h_n^{\beta}).
     \end{split}
\end{align}

For the second part,
\begin{align}\label{three parts}
 \begin{split}
        &\E \nm{\hgn-\E \hgn }_{\infty}\\
  = &\E \nm{\sum_{i=1}^n \br{f^2(i/n)(\xi_i^2-1)+(\tau_i^2-\sigma_n^2)+2f(i/n)\xi_i\tau_i}W_{ni}^*(x)}_{\infty}\\
  \le & \E \nm{\sum_{i=1}^n \br{f^2(i/n)(\xi_i^2-1)}W_{ni}^*(x)}_{\infty}+\E \nm{\sum_{i=1}^n \br{\tau_i^2-\sigma_n^2}W_{ni}^*(x)}_{\infty}\\
   & +\E \nm{\sum_{i=1}^n \br{2f(i/n)\xi_i\tau_i}W_{ni}^*(x)}_{\infty}\\
   = & \E \nm{\sum_{i=1}^n \br{f^2(i/n)(\xi_i^2-1)}W_{ni}^*(x)}_{\infty}+\sigma_n^2\E \nm{\sum_{i=1}^n \br{\tau_i^2/\sigma_n^2-1}W_{ni}^*(x)}_{\infty}\\
   & +\sigma_n \E \nm{\sum_{i=1}^n \br{2f(i/n)\xi_i\tau_i/\sigma_n}W_{ni}^*(x)}_{\infty}
 \end{split}
\end{align}

Note that each of the three terms above can be written as $\abs{\sum_{i=1}^n\zeta_i W_{ni}^*(x)}$ where $\zeta_i,i=1,2,..,n$ represents everything except for $W_{ni}^*(x)$ part. As the first step in finding an upper bound for the above three terms, we plug in $W_{ni}^*(x)$ given in \eqref{eq:W*def}, and simplify the expression in the following way:

\begin{align}\label{eq:zeta_idef}
\begin{split}
      \abs{\sum_{i=1}^n\zeta_i W_{ni}^*(x)} &= \frac{1}{nh_n} \left|\sum_{i=1}^n \zeta_i \mathbf{U}^T(0)\mathcal{B}_{nx}^{-1}\mathbf{U}\pr{\frac{i/n-x}{h_n}}K\pr{\frac{i/n-x}{h_n}}\right| \\
  &\overset{(a)}{=} \frac{1}{nh_n} \left|\sum_{i=1}^n \zeta_i \mathbf{U}^T(0)\mathcal{B}_{nx}^{-1} S_i(x)\right| \overset{(b)}{\le} \frac{1}{nh_n}\nm{\mathcal{B}_{nx}^{-1}\sum_{i=1}^n \zeta_i S_i(x)}_2  \\
  &\overset{(c)}{\le} \frac{1}{\lambda_0nh_n}\nm{\sum_{i=1}^n \zeta_iS_i(x)}_2,
\end{split}
\end{align}
where we have used the following facts to obtain the above relations:
\begin{itemize}
\item Equality (a): To obtain this equality we have defined: 
\begin{align}
S_i(x):=\mathbf{U}\pr{i/n-x \over h_n}K\pr{i/n-x\over h_n}.
\end{align}
\item Inequality (b): This inequality is an application of Cauchy-Schwartz inequality. 
\item Inequality (c): To obtain this inequality we have used part (4) of  \cref{basic properties of W} and have set $\lambda_0= \frac{1}{C_*}$
\end{itemize}
We should mention that the above steps have been used in Section 1.6.2 of \cite{Nonp} for the classical nonparametric regression.

Since the function $K$ and all components of $\mathbf{U}$ are bounded and Lipchitz on $[-1,1]$, there exists $L_{S}>0$ such that
\begin{align}
    \|S_i(x)-S_i(x')\|_2\le \frac{L_s}{h_n}|x-x'|
\end{align}

Now set $M=n^2$ and define $x_j:=j/M$ for $j=1,2,...,M$. We have

\begin{align}\label{switching to max}
   \begin{split}
          |\sum_{i=1}^n &\zeta_i W_{ni}^*(x)| \leq  \frac{1}{\lambda_0nh_n}\sup_{x\in [0,1]}\nm{\sum_{i=1}^n \zeta_iS_i(x)}_2  \\
 \le &\frac{1}{\lambda_0 nh_n}\pr{\max_{1\le j \le M}\nm{\sum_{i=1}^n \zeta_iS_i(x_j)}_2+\sup_{|x-x'|\le 1/M}\nm{\sum_{i=1}^n \zeta_i\pr{S_i(x)-S_j(x')}}_2} \\
 \le &\frac{1}{\lambda_0 \sqrt{nh_n}}\max_{1\le j \le M}\nm{\frac{1}{\sqrt{nh_n}}\sum_{i=1}^n \zeta_iS_i(x_j)}_2+\frac{L_s}{\lambda_0 Mnh_n^2}\sum_{i=1}^n |\zeta_i|  \\
 = & \frac{1}{\lambda_0 \sqrt{nh_n}}\max_{1\le j \le M}\nm{Z_j}_2+\frac{L_s}{\lambda_0 Mnh_n^2}\sum_{i=1}^n |\zeta_i|,
   \end{split}
\end{align}
where to obtain the last equality we have used the definition:
\begin{align}
    Z_j:=&\frac{1}{\sqrt{nh_n}}\sum_{i=1}^n \zeta_iS_i(x_j). 
\end{align}
Note that for our choice of $M$, we have
\begin{align}
    \E\br{\frac{L_s}{\lambda_0 Mnh_n^2}\sum_{i=1}^n |\zeta_i|}=O\pr{1\over n^2h_n^2}.
\end{align}
Now for each $1\le j\le M$, define 
\begin{align}
  \begin{split}
        Z_j:=&\frac{1}{\sqrt{nh_n}}\sum_{i=1}^n \zeta_iS_i(x_j)\\
    =&\frac{1}{\sqrt{nh_n}}\sum_{i=1}^n \zeta_i\mathbf{U}\pr{i/n-x_j \over h_n}K\pr{i/n-x_j\over h_n} 
  \end{split}
\end{align}

We will show (after choosing an optimized $h_n$) that $\E (\max_{1 \leq j \leq M} \|Z_j\|_2 )$ grows at the rate of $\sqrt{\log n}$ by first bounding $\E (\max_{1 \leq j \leq M} \|Z_j\|_2^2)$. To this end, we use a truncation trick for $Z_j$ and Bernstein's concentration inequality as we will describe below.

For fixed $j$, and $1\le r\le k$, let $u_{i,r}=u_{i,j,r}$ denote the $r$-th component of the vector $\mathbf{U}\pr{i/n-x_j \over h_n}$.  
Define the event 
\begin{align}
    \mathcal E_n(t):=\bigcup_{j=1}^M  \bigcup_{r=1}^k \cb{\abs{\sum_{i=1}^n \zeta_iu_{i,r} K\pr{i/n-x_j\over h_n}}\ge \frac{t}{r!}}
\end{align}

Comparing \eqref{three parts} and \eqref{eq:zeta_idef} we can see that for obtaining an upper bound for \eqref{three parts} we need to bound $\nm{\sum_{i=1}^n \zeta_iS_i(x)}_2$ for the following three choices of $\zeta_i$: $\zeta_i=f^2(i/n)(\xi_i^2-1), \tau_i^2/\sigma_n^2-1$ and $f(i/n)\xi_i\tau_i/\sigma_n$. It is straightforward to see that for these three choices there exists an absolute constant $c_{\zeta}$ such that 
$\xi_i, i=1,2,...,n$ is a sub-exponential random variable with sub-exponential norm bounded by $c_{\zeta}$.\footnote{Indeed, let $C_{sn}$ denote the sub-gaussian norm of a standard normal Gaussian random variable, namely $\|\xi_i\|_{\rm subgau}=C_{sn}$. From \cref{basic properties of subgau and subexp} we have, $    \|\tau_i^2/\sigma_n^2-1\|_{\rm subexp}=\|\xi_i^2-1\|_{\rm subexp}\le C\|\xi_i^2\|_{\rm subexp}$,  where $ C\|\xi_i^2\|_{\rm subexp} \le C\|\xi_i\|_{\rm subgau}^2=CC_{sn}^2$. Similarly, $ 
\|\xi_i^2\tau_i^2/\sigma_n^2\|_{\rm subexp}\le \|\xi_i\|_{\rm subgau}\|\tau_i/\sigma_n\|_{\rm subgau}=C_{sn}^2$. 
}

\begin{lemma}\label{lem:proof:subexp}
There exists an absolute constant, $c_{\zeta}$, that does not depend on $n$, such that all the random variables $f^2(i/n)(\xi_i^2-1), \tau_i^2/\sigma_n^2-1$, and $f(i/n)\xi_i\tau_i/\sigma_n$ for $i=1, 2, \ldots, n$ are sub-exponential random variables with sub-exponential norm bounded by $c_{\zeta}$.
\end{lemma}
While the proof of the claim is straightforward, we include it in \cref{ssec:proof:lemma:subexpterms} for completeness. Now by Bernstein's inequality \cref{berstein's inequality for sub-exponential distribution}, we have 
\begin{align}\label{bound using Berstein's inequality}
 \begin{split}
        & \P \pr{\abs{\sum_{i=1}^n \zeta_iu_{i,r} K\pr{i/n-x_j\over h_n}}\ge \frac{t}{r!}} \\
    \overset{(a)}{\le} & \exp\pr{-c\min\cb{\frac{t^2/r!^2}{c_{\zeta}^2\sum_{i=1}^n (\frac{1}{r!})^2K^2\pr{i/n-x_j\over h_n}},\frac{t/r!}{c_{\zeta}\max_{1\le i \le n}\frac{1}{r!}K\pr{i/n-x_j\over h_n}}}} \\
    \overset{(b)}{\le} & \exp\pr{-c\min\cb{\frac{2t^2}{c_{\zeta}^2nh_n},\frac{2t}{c_{\zeta}}}},
 \end{split}
\end{align}
where to obtain Inequality (a) we have used \cref{berstein's inequality for sub-exponential distribution}, \cref{lem:proof:subexp} and the fact that since $K$ is supported on $[-1,1]$, $\abs{u_{i,r}K\pr{i/n-x_j\over h_n}}\le \frac{1}{r!}K\pr{i/n-x_j\over h_n}$, and
\begin{align*}
\sum_{r=1}^ku_{i,r}^2K^2\pr{i/n-x_j\over h_n}\le 
    \sum_{r=1}^k\frac{1}{(r!)^2}K^2\pr{i/n-x_j\over h_n}.
\end{align*}
To obtain Inequality (b) we have used the fact that by construction, $K\pr{i/n-x_j\over h_n}=0$ whenever $\abs{i/n-x_j\over h_n}\ge 1$, and by definition \cref{the kernel function}, $K(x)=2K^2(x)$. Hence,
\begin{align}
  \begin{split}
        &\frac{1}{nh_n}\sum_{i=1}^n K^2\pr{i/n-x_j\over h_n}=\frac{1}{2nh_n}\sum_{i=1}^n K\pr{i/n-x_j\over h_n} \\
    =&\frac{1}{4nh_n}\sum_{i=1}^n I(-h_n+x_j\le i/n\le h_n+x_j)
    \le  \frac{1}{4}\max\pr{2,\frac{1}{nh_n}} \le \frac{1}{2},
  \end{split}
\end{align}
where to obtain the last inequality we have used the fact that $nh_n\ge \frac{1}{2}$ as $n\to\infty$.

Using the union bound and \eqref{bound using Berstein's inequality} we can  conclude that
\begin{align}\label{event:tailbound}
    \P\pr{\mathcal E_n(t)}\le Mk \exp\pr{-c\min\cb{\frac{2t^2}{nh_nc_{\zeta}^2},\frac{2t}{c_{\zeta}}}}.
\end{align}
Note that on the complement event $\mathcal E_n(t)^c$, we have that for every $1\le j \le M$,
\begin{align}\label{bound of Z_j}
  \begin{split}
        \|Z_j\|_2^2:=&\frac{1}{nh_n}\br{
        \sum_{i=1}^n \zeta_iu_{i,1} K\pr{i/n-x_j\over h_n}}^2+\cdots+\frac{1}{nh_n}\br{\sum_{i=1}^n \zeta_iu_{i,k} K\pr{i/n-x_j\over h_n}}^2 \\
        \le & \pr{\sum_{r=1}^k \frac{1}{r!^2}}
        \frac{t^2}{nh_n}
        \le \frac{2t^2}{nh_n}.
  \end{split}
\end{align}
Our next goal is to use the tail bound obtained in \eqref{event:tailbound} for obtaining an upper bound for $\E\br{\max_{1\le j\le M} \nm{Z_j}_2^2}$. We have 
\begin{align}\label{eq:E:Z_j:max}
   \begin{split}
        &\E\br{\max_{1\le j\le M}\nm{Z_j}_2^2}\\
        =&\E\br{\max_{1\le j\le M} \nm{Z_j}_2^2\mathbf{1}_{\mathcal E_n(t)^c}}+\E\br{\max_{1\le j\le M} \nm{Z_j}_2^2\mathbf{1}_{\mathcal E_n(t)}}  \\
    =&\E\br{\max_{1\le j\le M} \nm{Z_j}_2^2\mathbf{1}_{\mathcal E_n(t)^c}}+\int_0^{\infty}\P\br{\max_{1\le j\le M} \nm{Z_j}_2^2\mathbf{1}_{\mathcal E_n(t)}>s}ds  \\
    \overset{(a)}{\le }&\frac{2t^2}{nh_n}+\int_{2t^2 \over nh_n}^{\infty}\P\br{\max_{1\le j\le M} \nm{Z_j}_2^2>s}ds
   \end{split}
\end{align}

Here $(a)$ follows from \cref{bound of Z_j}. Hence, the main remaining step is to obtain an upper bound for $\int_{2t^2 \over nh_n}^{\infty}\P\br{\max_{1\le j\le M} \nm{Z_j}_2^2>s}ds$.

For any $s\ge 0$, it follows from  \cref{bound using Berstein's inequality} and the union bound that we have the tail probability bound
\begin{align}
\begin{split}
    & \P\br{\max_{1\le j\le M} \nm{Z_j}_2^2>s} \le  \sum_{j=1}^M \P\br{\nm{Z_j}_2^2>s} \\
\le & \sum_{j=1}^M \sum_{r=1}^k \P \pr{\frac{1}{nh_n}\abs{\sum_{i=1}^n \zeta_iu_{i,r} K\pr{i/n-x_j\over h_n}}^2> \frac{s}{4r!^2}}\\
\le & \sum_{j=1}^M \sum_{r=1}^k \P \pr{\frac{1}{\sqrt{nh_n}}\abs{\sum_{i=1}^n \zeta_iu_{i,r} K\pr{i/n-x_j\over h_n}}> \frac{\sqrt{s}}{2r!}}\\
\le & Mk\exp\pr{-c\min\cb{\frac{s}{2c_{\zeta}^2},\frac{\sqrt{nh_ns}}{c_{\zeta}}}},
\end{split}
\end{align}
where the last equality is obtained from \eqref{bound using Berstein's inequality}. Using this expression we now are in a position to bound $\int_{2t^2 \over nh_n}^{\infty}\P\br{\max_{1\le j\le M} \nm{Z_j}_2^2>s}ds$.

Setting $\frac{s}{2c_{\zeta}^2}=\frac{\sqrt{snh_n}}{c_{\zeta}}$, we have $s=c'nh_n$ where $c'=4c_{\zeta}^2$. By splitting the integral $\int_{2t^2 \over nh_n}^{\infty}\P\br{\max_{1\le j\le M} \nm{Z_j}_2^2>s}ds$ at $s=c'nh_n$ we have:

\begin{align}\label{eq:int:max:Z_j:prob}
\begin{split}
        & \int_{2t^2 \over nh_n}^{\infty}\P\br{\max_{1\le j\le M} \nm{Z_j}_2^2>s}ds  \\
\le & Mk \pr{\int_{2t^2 \over nh_n}^{c'nh_n}\exp\pr{-c\frac{s}{2c_{\zeta}^2}}ds+ \int_{c'nh_n}^{\infty}\exp\pr{-c\frac{\sqrt {nh_ns}}{c_{\zeta}}}ds}  \\
= & Mk \br{\frac{2c_{\zeta}^2}{c}\exp\pr{-c\frac{2t^2/nh_n}{2c_{\zeta}^2}}-\frac{2c_{\zeta}^2}{c}\exp\pr{-c\frac{c'nh_n}{2c_{\zeta}^2}}} \\
+ & Mk\cdot  \pr{\frac{\sqrt{c'}c_{\zeta}}{c}} \cdot \exp\pr{-\frac{c}{c_{\zeta}}\sqrt{c'}nh_n} + Mk\cdot  \pr{\frac{c^2_{\zeta}}{c^2 n h_n}} \cdot \exp\pr{-\frac{c}{c_{\zeta}}\sqrt{c'}nh_n},
\end{split}
\end{align}
where to obtain the last inequality we have used the following integration formulas:
\begin{align}
  \begin{split}
        &\int_A^{\infty} \exp \pr{-\lambda s}ds=\frac{\exp\pr{-\lambda A}}{\lambda},  \\
    &\int_A^{\infty}\exp\pr{-\lambda\sqrt{s}}ds=\pr{\frac{\sqrt{A}}{\lambda}+\frac{1}{\lambda^2}}\exp(-\lambda \sqrt{A})
  \end{split}
\end{align}

\vspace{3mm}
Choosing $t=C_k \sqrt{nh_n\log n}$ for sufficiently large constant $C_k>0$, and combining \eqref{eq:E:Z_j:max} and \eqref{eq:int:max:Z_j:prob} we have
\begin{align*}
    \E\br{\max_{1\le j\le M} \nm{Z_j}_2^2}=O\pr{\log n +M\exp\pr{-Cnh_n}}. 
\end{align*}
Hence,
%$\eta_i:=\zeta_i^2-\E\br{\zeta_i^2}$ with $\zeta_i=\xi_i^2-1, \tau_i^2/\sigma_n^2-1$ and $\xi_i^2\tau_i^2/\sigma_n^2$, and $m=M=n^l$ for sufficiently large $l$  yields \textcolor{blue}{Now the previous $\log^2 n$ issue seems to be fixed. But still, the most ideal form is $\log \frac{n}{\max(1,\sigma_n^4)}$. Our bound is only sharp when $\sigma_n^4=O(n^{1-\eta})$. If $\sigma_n^4=n/\log n$ then not good. This could be the limitation of our estimator. Also at the beginning we singled our $\sigma_n^2$ which limits our freedom of putting $\sigma_n$ inside $\log(n)$}

\begin{align*}
\E\br{\max_{1\le j\le M}\|Z_j\|_2}\le \sqrt{\E\br{\max_{1\le j\le M}\|Z_j\|_2^2}}=O\pr{\sqrt{\log n}+\sqrt{M\exp\pr{-Cnh_n}}}.
\end{align*}

 Plugging this back to \cref{switching to max} and combining three parts from \cref{three parts} yields the upper bound (recall that $M=n^2$)
\begin{equation*}
   \E \nm{\hgn-\E \hgn }_{\infty}=O\pr{ \max(1,\sigma_n^2)\sqrt{\frac{\log n}{nh_n}}+\max(1,\sigma_n^2)\sqrt{n^2\exp\pr{-Cnh_n}\over nh_n}}
\end{equation*}
Together with \cref{bias part infinity norm}, we have
\begin{align*}
\begin{split}
       &\E \br{\nm{\hgn-f^2}_{\infty}}\le O\pr{ \max(1,\sigma_n^2)\frac{\sqrt{\log n}}{\sqrt{nh_n}}+\max(1,\sigma_n^2)\sqrt{n^2\exp\pr{-Cnh_n}\over nh_n}+h_n^{-\beta}}
\end{split}
\end{align*}

Now by choosing $h_n=\pr{\max(1,\sigma_n^4)\frac{\log n}{n}}^{\frac{1}{2\beta+1}}$, we see that the middle term has a faster decay rate than the first term, and that 

\begin{equation}
    R_{\infty} \left(\Sigma_{\mathfrak{h}}(\beta, L), \sigma_n\right) \overset{(a)}{\le} \sup_{f\in \Sigma_{\mathfrak{h}}(\beta,L)} \frac{1}{\mathfrak{h}}\E \br{\nm{\hgn-f^2}_{\infty}}=O_{\beta,L,\mathfrak{h},\eta}\pr{\pr{\max(1,\sigma_n^4)\frac{\log n}{n}}^{\frac{\beta}{2\beta+1}}},
\end{equation}
where for $(a)$ we have again used the assumption that $f\ge \mathfrak{h}$ (c.f. \cref{passing from square to no square}).

\subsubsection{Proof of \cref{lem:proof:subexp}}\label{ssec:proof:lemma:subexpterms}

\section{Conclusion}\label{Conclusion}
In this paper, we presented the first minimax analysis of recovering a smooth function $f$ from its samples that are corrupted with speckle noise. 
More specifically, let \(n\) represents the number of data points, \(f\) be a \(\beta\)-H{\"o}lder continuous function, and \(\hat{\nu}_n\) be an estimator of \(f\) obtained from the samples of $f$ corrupted by the speckle and additive noises. We proved that the minimax risk can then be expressed as \(\inf_{\hat{\nu}_n} \sup_f \mathbb{E}_f\|\hat{\nu}_n - f \|^2_2 = \Theta_{\beta, L}\pr{ \min\cb{1,\pr{\frac{\max(1,\sigma_n^4)}{n}}^{\frac{2\beta}{2\beta+1}}}}\). 

%While the minimax analysis is mainly concerned with large values of $n$, our simulation results have confirmed that this rate is achieved for practically relevant values of $n$. 

\newpage
\appendix
\section{Rationale for De-Speckling Modeling Assumptions}\label{app:model_justification}

 In this section, we describe key properties of the de-speckling problem that motivated our choice of the function class $\Sigma_{\mathfrak{h}}(\beta, L)$. From the model formulation,
 $$
 y_i = f(x_i)\, \xi_i + \tau_i, \quad i = 1, 2, \ldots, n,
 $$
 it is evident that the symmetry of the noise introduces an identifiability issue: both $f$ and $-f$ produce identical observations. However, the identifiability challenges extend beyond this simple symmetry. For example, the functions $f(x) = x - \frac{1}{2}$ and $g(x) = \left|x - \frac{1}{2}\right|$, defined on $[0,1]$, also lead to identical observations. One natural constraint to ensure identifiability is to restrict the range of $f(\cdot)$ to be a subset of $[0, \infty)$.

In addition to the positivity constraint imposed in $\Sigma_{\mathfrak{h}}(\beta, L)$, we also assume that $f(x) \geq \mathfrak{h} > 0$. To motivate this assumption, consider a simplified setting in which we aim to estimate a scalar parameter $\theta_0 > 0$ from observations of the form:
$$
y_i = \theta_0 \xi_i + \tau_i, \quad i = 1, 2, \ldots, n.
$$
A straightforward calculation shows that the maximum likelihood estimator of $\theta_0$ is
$$
\hat{\theta} = \sqrt{\frac{1}{n} \sum_{i=1}^n y_i^2 - \sigma_n^2}.
$$
Applying the delta method, we obtain the asymptotic distribution:

$$
\sqrt{n}(\hat{\theta} - \theta_0) \overset{d}{\longrightarrow} \mathcal{N}\left(0, \frac{1}{4\theta_0^2} \cdot 2(\theta_0 + \sigma_n^2)^2\right).
$$
As this expression shows, when $\theta_0 \to 0$, the variance of the estimator diverges, and accurate estimation becomes infeasible. A similar issue arises in our setting: allowing the function $f(x)$ to approach zero can significantly degrade the performance of locally polynomial regression. For this reason, we define $\Sigma_{\mathfrak{h}}(\beta, L)$ to include functions satisfying $f(x) \geq \mathfrak{h}$ for some fixed constant $\mathfrak{h} > 0$.

 The final point we would like to make is that, in the definition of $\Sigma_{\mathfrak{h}}(\beta, L)$, we also assume that the function is upper bounded by 1. This assumption is natural in image processing applications, where pixel values—corresponding to each $f(i/n)$ in our model—are stored using a finite number of bits. Without loss of generality, We have scaled the function so that the upper bound is equal to $1$.

\section{Proof of \cref{mimimax error for the classical additive noise model}}\label{proof of mimimax error for the classical additive noise model}
In this appendix, we follow the approach used in \cite{Mainbook} and \cite{Nonp} to present the proof of \cref{mimimax error for the classical additive noise model}. While \cite{Mainbook} and \cite{Nonp} assume \( \sigma_n = \Theta(1) \), allowing certain terms to be neglected, our analysis also considers the regime \( \sigma_n \to \infty \), and $\sigma_n \rightarrow 0$ requiring us to account for these previously negligible terms. Aside from this distinction, the overall proof strategy closely follows the methods in \cite{Mainbook} and \cite{Nonp}.

\subsection{The $L_2$-norm case}

\subsubsection{Proof of the lower bound}\label{proof of the lower bound for the classical model in L_2 norm}

For obtaining lower bounds we consider the following three cases:

\begin{itemize}
\item $\sigma_n=O(n^{-\beta})$: Note that in this case \cref{mimimax error for the classical additive noise model} suggests that
\[
\widetilde{R}_2(\Sigma_{\mathfrak{h}}(\beta, L), \sigma_n) = \Theta (n^{-2\beta}).
\]
To prove this claim, first note that from the monotonicity of $\widetilde{R}_2(\Sigma_{\mathfrak{h}}(\beta, L), \sigma_n)$, i.e. \cref{lem:risk:monotonicity}, we have
\[
\widetilde{R}_2(\Sigma_{\mathfrak{h}}(\beta, L), \sigma_n) \geq \widetilde{R}_2(\Sigma_{\mathfrak{h}}(\beta, L), 0). 
\]
Hence, we prove a lower bound for the case $\sigma_n=0$. To obtain a lower bound in this case, we choose two hypotheses \begin{align}\label{two hypotheses}
    \begin{split}
        & f_{0,n}(x):\equiv \mathfrak{h}, \text{ and}\\
    & f_{1,n}(x):= \mathfrak{h}+\sum_{i=0}^{n-1} n^{-\beta}\phi\pr{\frac{x-(2i+1)}{2n}},
    \end{split}
\end{align}
where $\phi$ is the basic function constructed in \cref{explicit construction of basic function}. Similar to the proof of \cref{f and f^2 holder} it is straightforward to confirm that $f_{j,n}\in \Sigma(\beta, L)$ for $j=0,1$ and
\begin{align}
 \begin{split}
        \nm{f_{1,n}-f_{0,n}}_2^2 & = \frac{L^2 n^{-2\beta}}{\overline{\phi_0}^2}\sum_{i=1}^{n-1}\int_{i/n}^{(i+1)/n}\exp\pr{-\frac{8}{1-4n^2 \pr{x-\frac{2i+1}{2n}}^2}} dx  \\
    =& \frac{L^2 n^{-2\beta}}{\overline{\phi_0}^2}\cdot \frac{1}{2n}\sum_{i=0}^{n-1}\int_{-1}^{1}\exp\pr{-\frac{8}{1-x^2}} dx  \\
    =& \Theta_{\beta, L}\pr{n^{-2\beta}}
 \end{split}
\end{align}
Therefore

\begin{align}
  \begin{split}
    \inf_{\hat{f}_n} \sup_{f \in \Sigma_{\mathfrak{h}}(\beta,1)} \mathbb{E} \| \hat{f}_n - f \|_2^2 
    &\ge \inf_{\hat{f}_n} \max \left\{ \| \hat{f}_n - f_{1,n} \|_2^2, \| \hat{f}_n - f_{0,n} \|_2^2 \right\} \\
    &\ge \inf_{\hat{f}_n} \left( \max \left\{ \| \hat{f}_n - f_{1,n} \|_2, \| \hat{f}_n - f_{0,n} \|_2 \right\} \right)^2 \\
    &\ge \frac{1}{4} \inf_{\hat{f}_n} \left( \| \hat{f}_n - f_{1,n} \|_2 + \| \hat{f}_n - f_{0,n} \|_2 \right)^2 \\
    &\ge \frac{1}{4} \| f_{1,n} - f_{0,n} \|_2^2 
    = \Theta_{\beta, L} \left( n^{-2\beta} \right).
  \end{split}
\end{align}
 \item $\sigma_n=w(n^{-\beta})$ and $\sigma_n=o(n^{1/2})$:
In this case, our proof follows steps similar to those in Section 2.6.1 of \cite{Nonp}, with minor modifications to account for the variance of the noise.  

Set $m:=\left \lceil c_0 (n/\sigma_n^2)^{1 \over 2\beta+1} \right \rceil$ and $h_n:=\frac{1}{m}$. For $i=0,1,...,m-1$, we define 
\begin{align*}
    \varphi_i(x):=Lh_n^{\beta}\phi\pr{x-\frac{2i+1}{m} \over h_n}.
\end{align*}
where $\phi$ is the basic function constructed in \cref{explicit construction of basic function}. Set $M=\left \lceil 2^{m/8} \right \rceil$ and for $j=1, 2, \ldots, M$ construct the binary vectors $\omega^{(j)}:=(\omega_1^{(j)},...,\omega_m^{(j)})\in \cb{0,1}^m$ for $j=0,1,...,M$ (with $\omega^{(0)}:=(0,...,0)$) such that they satisfy
\begin{align*}
    \rho(\omega^{(j)},\omega^{(j')})\ge \frac{m}{8}.
\end{align*}
The existence of these vectors is guaranteed by the Gilbert-Varshamov bound (\cref{GV}).

Now we construct the functions $f_{j,n}(x)$ in the following way: 
\begin{align*}
    f_{j,n}(x):=\mathfrak{h}+f_{\omega^{(j)}}(x)=\mathfrak{h}+\sum_{i=0}^m \omega_i^{(j)}\varphi_{i}(x)
\end{align*}
Now if we show
\begin{align*}\label{critial lower bound for tail probability}
    \inf_{\hfn}\max_{f\in \cb{f_{0n},...,f_{Mn}}}\P\pr{\|\hfn-f\|_2^2\ge Ah_n^{-2\beta}}\ge C
\end{align*}
for some constant $C>0$, then it will follow from the Markov's inequality that
\begin{align*}
   \begin{split}
         &\inf_{\hfn} \sup_{f\in \Sigma_{\mathfrak{h}}(\beta, L)} \E\br{\|\hfn-f\|_2^2} \\
    \ge & Ah_n^{-2\beta}\inf_{\hfn}\max_{f\in \cb{f_{0,n},...,f_{M,n}}}\P\pr{\|\hfn-f\|_2^2\ge Ah_n^{-2\beta}}  \\
    =& \Omega(h_n^{-2\beta}).
   \end{split}
\end{align*}

Now, to establish \cref{critial lower bound for tail probability}, we use Theorem 2.5 from \cite{Nonp}. It suffices to verify two assumptions
\begin{enumerate}
    \item $\|f_{j,n}-f_{j',n}\|_2=\Omega_{\beta, L}\pr{ h_n^{-\beta}}$ for $0\le j< j'\le M$; and
    \item For any $1\le j \le M$,
    $\P_{f_{j,n}}\ll \P_{f_{0,n}}$ and 
    \begin{align}
        \frac{1}{M}\sum_{j=1}^M {\rm KL}(\P_{f_{j,n}},\P_{f_{0,n}})\le \alpha \log M
    \end{align}
    for some constant $\alpha>0$.

\vspace{3mm}
    The verification of these two conditions are given in the pages 103 and 106 of \cite{Nonp}, and we will not repeat it here.
\end{enumerate}

\item $\sigma_n = \Omega(n^{1/2})$. 
In the previous step,  for any $\tilde \sigma_n=o(n^{1/2})$, we have established the following lower bound:$\tilde{R}_2(\Sigma_{\mathfrak{h}}(\beta,L),\tilde \sigma_n)= \Omega \pr{\pr{\frac{\tilde \sigma_n^2}{n}}^{\frac{2\beta}{2\beta+1}}}.$ We claim that this forces $\tilde{R}_2(\Sigma_{\mathfrak{h}}(\beta,L),\sigma_n)=\Omega(1)$ for $\sigma_n = \Omega(\sqrt{n})$. Suppose for contradiction that $$u_n:=\tilde{R}_2(\Sigma_{\mathfrak{h}}(\beta,L),\sigma_n)\to 0.$$
Choose $\tilde \sigma_n^2 =nu_n^{\beta+1 \over 2\beta}$. Note that 
\begin{align*}
    \pr{\frac{\tilde \sigma_n^2}{n}}^{\frac{2\beta}{2\beta+1}}=u_n^{\beta+1 \over 2\beta+1}\gg u_n.
\end{align*}
Hence, using the monotonicity of $\tilde{R}_2(\Sigma_{\mathfrak{h}}(\beta,L),\sigma_n)$ as a function of $\sigma_n$ proved in \cref{lem:risk:monotonicity} we should have: 
\begin{align*}\label{A:contradict:squarelower1}
\tilde{R}_2(\Sigma_{\mathfrak{h}}(\beta,L),\tilde \sigma_n) \leq \tilde{R}_2(\Sigma_{\mathfrak{h}}(\beta,L),\sigma_n)=u_n
\end{align*}
However, from the previous section since $\tilde{\sigma}_n=w(n^{-\beta})$ and $\tilde{\sigma}_n=o(n^{1/2})$
\begin{align*}\label{A:contradict:squarelower2}
\tilde{R}_2(\Sigma_{\mathfrak{h}}(\beta,L),\tilde{\sigma}_n) = \Omega\pr{\left(\frac{\tilde \sigma_n^2}{n}\right)^{\frac{2\beta}{2\beta+1}}} =\Omega\left(u_n^{\beta+1 \over 2\beta+1} \right). 
\end{align*}
Note that \eqref{A:contradict:squarelower2} and \eqref{A:contradict:squarelower1} are in contradiction with each other. 
\end{itemize}

\subsubsection{Proof of the upper bound}\label{Proof of the upper bound for the L_2 case}
To obtain an upper bound for $R_2(\Sigma_{\mathfrak{h}}(\beta,L), \sigma_n)$, following Section 1.6 of \cite{Nonp}, we also consider the local polynomial estimator 
\begin{align*}
    \hfn:=\sum_{i=1}^n Y_iW_{ni}^*(x)
\end{align*}
where $W_{ni}^*$ is defined in \cref{eq:W*def}. We have 
\begin{align*}
     \mathbb{E}_f \|f- \hat{f}_n\|_2^2=\uint b^2(x)dx+\uint \sigma^2(x)dx,
\end{align*}
where
\begin{align*}
    b(x):=\E_f\br{\hat{f}_n(x)}-f(x) \quad \text{ and } \quad \sigma^2(x):=\E_f\br{\pr{\hat{f}_n(x)-\E_f\br{\hat{f}_n(x)}}^2}
\end{align*}
Combining the result of Lemma 1.5 and Proposition 1.13 of \cite{Nonp} leads to the following lemma:
\begin{lemma}\label{lem:tsybakov:result}
If as $n \rightarrow \infty$, $h_n \rightarrow 0$ and $n h_n \rightarrow \infty$, then there exists two constants $q_1, q_2$ (not depending on $n$, but can depend on $\beta$ and $L$) such that 
\begin{align}
\begin{split}
    b^2(x) &\leq q_1 h_n^{2 \beta},  \\
\sigma^2(x) &\leq q_2 \frac{\sigma_n^2}{n h_n}. 
\end{split}
\end{align}
\end{lemma}

Now based on this result we consider the following three case:

\begin{itemize}
    \item $\sigma_n=\Omega(n^{1/2})$: In this case, by choosing the estimator function $\hfn=0$, it is trivial to confirm that:
    \[
    \tilde{R}_2 (\Sigma_{\mathfrak{h}}(\beta, L), \sigma_n) \leq \sup_{f \in \Sigma_{\mathfrak{h}}(\beta, L)} \mathbb{E}\|0- f\|_2^2 \leq 1. 
    \]

    \item $\sigma_n=\omega(n^{-\beta})$,  but $\sigma_n=o(n^{-1/2})$: In this case, we set $h_n:=\pr{\sigma_n^2 \over n}^{1 \over 2\beta+1}$. It is straightforward to confirm that this choice of $h_n$ satisfies the requirements of \cref{lem:tsybakov:result}. Hence,
\begin{align}
  \begin{split}
        &\tilde{R}_2 (\Sigma_{\mathfrak{h}}(\beta, L), \sigma_n)\\ \leq & \sup_{f \in \Sigma1(\beta, L)} \mathbb{E}\|\hat{f}_n(x)- f\|_2^2  \\
    =& \uint b^2(x)dx+\uint \sigma^2(x)dx  \\
    \leq & q_1 \pr{\sigma_n^2 \over n}^{2 \beta \over 2\beta+1} + q_2 \pr{\sigma_n^2 \over n}^{2 \beta \over 2\beta+1} = O\pr{\pr{\sigma_n^2 \over n}^{2\beta \over 2\beta+1}}. 
  \end{split}
    \end{align}

    \item $\sigma_n=O( n^{-\beta})$: Recall that in this case we established a lower bound of  $\Theta(n^{-2\beta})$ for the minimax MSE. Suppose that for some $\sigma_n=O( n^{-\beta})$ we have
    \[
    u_n := \frac{R_2 (\Sigma_{\mathfrak{h}}(\beta, L), \sigma_n^2)}{n^{-2\beta}} \rightarrow \infty. 
    \]
    Define $\tilde{\sigma}_n = n^{-\beta}u_n$. Note that $\tilde{\sigma}_n = \omega(n^{-\beta})$. Hence, by using the locally linear polynomial estimator we can prove that 
    \[
    \tilde{R}_2(\Sigma_{\mathfrak{h}}(\beta, L), \tilde{\sigma}_n) = O\left(\left(\frac{\tilde{\sigma}_n^2}{n}\right)^{\frac{2 \beta}{2 \beta+1}}\right) = O\left(n^{-2\beta} u_n^{\frac{2 \beta}{2 \beta+1}}\right). 
    \]
Hence, since $u_n \rightarrow \infty$ we can conclude that
\[
\tilde{R}_2 (\Sigma(\beta, L), \tilde{\sigma}_n) < \tilde{R}_2 (\Sigma(\beta, L), {\sigma}_n). 
\]
However, this is in contradiction with the non-decreasing nature of the minimax risk proved in \cref{lem:risk:monotonicity}. 

   \iffalse 
    On the other hand, by the fact that the error is a monotone non-decreasing function of the variance $\sigma_n$, we must have
    \begin{align}
        \sup_{f\in \Sigma_{\mathfrak{h}}(\beta, L)}\E\|\hfn-f\|_2^2=O\pr{\pr{w_n^2 \over n}^{\frac{2\beta}{2\beta+1}}}
    \end{align}
    for any $w_n=\omega(n^{-\beta})$. We claim that this forces $\sup_{f\in \Sigma_{\mathfrak{h}}(\beta, L)}\E\|\hfn-f\|_2^2=O\pr{n^{-2\beta}}$. Indeed, if on the contrary we have at the noise level $\sigma_n=O(n^{-\beta})$ that
    \begin{align}
        \sup_{f\in \Sigma_{\mathfrak{h}}(\beta, L)}\E\|\hfn-f\|_2^2=\eta_n \cdot n^{-2\beta}
    \end{align}
     with $\limsup_{n\to \infty}\eta_n=+\infty$. By passing to a subsequence we assume the $\lim_{n\to \infty}\eta_n=+\infty$.
     Then as $\eta_n \cdot n^{-2\beta}=\pr{\frac{\eta_n^{2\beta+1 \over 2\beta}/n^{2\beta}}{n}}^{\frac{2\beta}{2\beta+1}}$, we have a strictly higher error rate than the error in the noise level $w_n=\eta_n^{1/2}/n^{\beta}=\omega(n^{-\beta})$. This contradicts the monotonicity of the error rate as a function of variance. Therefore, we must have
     \begin{equation}
         \sup_{f\in \Sigma_{\mathfrak{h}}(\beta, L)}\E\|\hfn-f\|_2^2=O(n^{-2\beta})
     \end{equation}
     \fi
\end{itemize}

\subsection{The $L_{\infty}$-norm case}
\subsubsection{Proof of the lower bound} 
For obtaining the lower bound, again we consider the following two different cases: 
\begin{itemize}
\item $\sigma_n = O((\log n)^{-1/2}n^{-\beta})$: As in the case with the $L_2$-norm, by monotonicity (\cref{lem:risk:monotonicity}), it suffices to establish the lower bound rate $n^{-\beta}$ for $\sigma_n=0$. In this case, we use the same hypotheses $f_{0,n}$ and $f_{1,n}$ as in \cref{two hypotheses}. It follows that 
\begin{align*}
    \nm{f_{1,n}-f_{0,n}}_{\infty}=\Theta(n^{-\beta}),
\end{align*}
and therefore,
% \begin{align*}
%     \inf_{\hfn} \sup_{f\in \Sigma_{\mathfrak{h}}(\beta, L)} \E_f\|\hfn-f\|_{\infty}
%     =\Omega({n^{-\beta}}).
% \end{align*}

\begin{align*}
   \begin{split}    &\inf_{\hfn}\sup_{f\in\Sigma_{\mathfrak{h}(\beta,L)}}\E \nm{\hfn-f}_{\infty} \\
\ge &\inf_{\hfn} \max \cb{ \nm{\hfn-f_{1,n}}_{\infty}, \nm{\hfn-f_{0,n}}_{\infty}} \\
\ge &  \frac{1}{2}\nm{f_{1,n}-f_{0,n}}_{\infty} =\Omega\pr{n^{-\beta}}.
   \end{split}
\end{align*}

\item $\sigma_n = \omega\pr{(\log n)^{-1/2}n^{-\beta}}$ and $\sigma_n = o\pr{n^{1-\eta \over 2}}$: 
%\item $\sigma_n = \Omega(n^{1/2})$: \textcolor{blue}{this is the base we are not able to connect from the previous case}

We shall follow the same approach as in Section 2.6.2 of \cite{Nonp}. Construct $M+1$ hypothesis $f_{j,n}$ for $j=0,1,...,M$ as 
\begin{align*}
    \begin{split}
        &f_{0,n}:\equiv \mathfrak{h},\\
    &f_{j,n}:=\mathfrak{h}+h_n^{\beta}\phi\pr{x-(2j-1)/M \over h_n}
    \end{split}
\end{align*}
where $\phi$ is the basic function defined in \cref{explicit construction of basic function}, $h_n=\pr{\log \frac{n}{\sigma_n^2}\middle/ \frac{n}{\sigma_n^2}}^{\frac{1}{2\beta+1}}$ and $M=\left \lceil c_0 h_n^{-1}\right \rceil$. Now $f_{j,n}\in \Sigma_{\mathfrak{h}}(\beta, L)$. For $0\le j<j'\le M$, we have
\begin{align*}
    \|f_{j,n}-f_{j',n}\|_{\infty}= \Omega_{\beta, L}\pr{h_n^{\beta}}.
\end{align*}
Now following the same argument as in the page 109 of \cite{Nonp}, there exists $\alpha\in \pr{0,\frac{1}{8}}$
\begin{align*}
    & \frac{1}{M}\sum_{j=1}^M {\rm KL}(\P_{f_{j,n}},\P_{f_{0,n}})\le \alpha \log M
\end{align*}
for sufficiently large $c_0>0$. By Theorem 2.5 from \cite{Nonp} and Markov's inequality, we have
\begin{align*}
    \inf_{\hfn} \sup_{f\in \Sigma_{\mathfrak{h}}(\beta, L)} \E_f\|\hfn-f\|_{\infty}
    =\Omega_{\beta, L}\pr{h_n^{\beta}}
\end{align*}
In the case when $\sigma_n^2=O\pr{n^{1-\eta}}$ for some $\eta\in (0,1)$, the bound reduced to $\Omega_{\eta, \beta, L}\pr{\sigma_n^2 \log n \over n}.$
\end{itemize}

\subsubsection{Proof of the upper bound}

\begin{itemize}
\item $\sigma_n = \omega\pr{(\log n)^{-1/2}n^{-\beta}}$ and $\sigma_n = o\pr{n^{1-\eta \over 2}}$:
Following the approach in Section 1.6.2 of \cite{Nonp}, 
\begin{align*}
    \E\nm{\hfn-f}_{\infty}^2
 \le 2\E\nm{\hfn-\E \hfn}_{\infty}^2 + 2\nm{\E\hfn-f}_{\infty}^2 
\end{align*}
Under the assumption that $h_n\to 0$ and $nh_n\to \infty$ and using the local polynomial estimator $\hfn$, we have
\begin{align*}
    \nm{\E\hfn-f}_{\infty}^2=h^{-2\beta},
\end{align*}
and
\begin{align*}
    \E\nm{\hfn-\E \hfn}_{\infty}^2=O_{\beta,L}\pr{\frac{\sigma_n^2\log M}{nh_n}+ \frac{\sigma_n^2}{M^2h_n^4}}. 
\end{align*}

we obtain the estimate under local polynomial estimator $\hfn$:
\begin{align}\label{three term estimation for the l infinity norm case}
    \begin{split}
        \E\br{\|\hfn-f\|^2_{\infty}}=& O_{\beta,L}\pr{\max\pr{2,\frac{1}{nh_n}}\frac{\sigma_n^2\log M}{nh_n}+ \frac{\sigma_n^2}{M^2h_n^4} + h_n^{2\beta}}\\
    =& O_{\beta,L}\pr{\frac{\sigma_n^2\log M}{nh_n}+ \frac{\sigma_n^2}{M^2h_n^4} + h_n^{2\beta}}. 
    \end{split}
\end{align}

%In the case of $\sigma_n=0$, we can choose $h_n=\frac{1}{2}$ and in this case the upper bound is $O\pr{n^{-2\beta}}$. 

Choose $M=(n/\sigma_n^2)^r$. Then the estimate becomes
\begin{align*}
 \begin{split}
     & O_{\beta,L}\pr{\frac{r\sigma_n^2\log (n/\sigma_n^2)}{nh_n}+ \frac{\sigma_n^2}{(n/\sigma_n^2)^{2r} h_n^4} + h_n^{2\beta}}\\
=& O_{\beta,L}\pr{\frac{r\sigma_n^2\log (n/\sigma_n^2)}{nh_n}+ \frac{\sigma_n^{4r+2}}{n^{2r} h_n^4} + h_n^{2\beta}}
 \end{split}
\end{align*}
We first optimize the first and last terms by choosing 
\begin{align*}
\tilde h_n:=\pr{r\log(n/\sigma_n^2)\over n/\sigma_n^2}^{1 \over 2\beta+1}
\end{align*}
In the case of $\sigma_n^2=o(n^{1-\eta})$ and $\sigma_n^2=\omega\pr{(\log n)^{-1}n^{-2\beta}}$, we have $\tilde h_n\to 0$ and $n\tilde h_n\to \infty$, respectively.

To optimize this three-term sum, we note that by plugging $\tilde h_n$ into the first and the third terms together (they have the same decay rate at $\tilde h_n$) give the upper bound for the sum these two parts as
\begin{align*} O_{\beta,L}\pr{h_n^{2\beta}}=O_{\beta,L}\pr{\pr{r\log(n/\sigma_n^2)\over n/\sigma_n^2}^{2\beta \over 2\beta+1}}
\end{align*}

By our assumption that $\sigma_n^2=o(n^{1-\eta})$ for some $\eta\in (0,1)$, we evaluate the middle term at this $\tilde h_n$ and compare it to $\tilde h_n^{2\beta}$:

\begin{align*}
 \frac{\sigma_n^{4r+2}}{n^{2r} \tilde h_n^4}   \cdot \tilde h_n^{-2\beta} =\sigma_n^2 \cdot \frac{\sigma_n^{4r}}{n^{2r}}\cdot \frac{1}{\tilde h_n^{2\beta+4}}=\sigma_n^2\cdot \frac{\sigma_n^{4r}}{n^{2r}}\cdot \pr{n/\sigma_n^2 \over r\log(n/\sigma_n^2) }^{2\beta+4 \over 2\beta+1}
\end{align*}
Since $\sigma_n^2=O(n^{1-\eta})$ for some $\eta\in (0,1)$, there exists a sufficiently large $r>0$ that only depends on $\eta$ and $\beta$ such that this quotient decays to zero, so that the middle term from \cref{three term estimation for the l infinity norm case} is negligible. Therefore in this case we have the upperbound 
\begin{align*}
& \E\nm{\hfn-f}_{\infty}\le \pr{\E\nm{\hfn-f}_{\infty}^2}^{1/2}\\
=& O_{\beta,L}\pr{\pr{\log(n/\sigma_n^2)\over n/\sigma_n^2}^{\beta \over 2\beta+1}}=O_{\eta, \beta,L}\pr{\pr{\log n \over n/\sigma_n^2}^{\beta \over 2\beta+1}}
\end{align*}

\item $\sigma_n=O\pr{(\log n)^{-1/2}n^{-\beta}}$: Recall that at this noise level, we already have a lower bound at the rate of $n^{-\beta}$, and for any $\sigma_n=\omega\pr{(\log n)^{-1/2}n^{-\beta}}$ and $\sigma_n = o\pr{n^{1-\eta \over 2}}$, we have an upper bound at the rate of $\pr{\log n \over n/\sigma_n^2}^{\beta \over 2\beta+1}$. Arguing in the same way as in the $L_2$-norm case with \cref{lem:risk:monotonicity} gives the upperbound $O_{\beta,L}(n^{-\beta})$ as desired.

\vspace{3mm}
Combining these, we have for $\sigma_n^2=O(n^{1-\eta})$ the upperbound
\begin{align*}
\begin{split}
    & O_{\beta,L}\pr{\max\cb{n^{-\beta}, \pr{\log(n/\sigma_n^2)\over n/\sigma_n^2}^{\beta \over 2\beta+1}}}\\
=&O_{\eta,\beta,L}\pr{\max\cb{n^{-\beta}, \pr{\log n \over n/\sigma_n^2}^{\beta \over 2\beta+1}}}.
\end{split}
\end{align*}
\end{itemize}

\subsection*{Acknowledgment}
We would like to thank anonymous referees for their careful reading and valuable feedback, which led to significant improvements in this work. Arian Maleki is supported in part by ONR award no. N00014-23-1-2371.

%\appendix
%\section{An example appendix} 

%\section*{Acknowledgments}
%We would like to acknowledge the assistance of volunteers in putting
%together this example manuscript and supplement.

%\bibliographystyle{siamplain}
\bibliographystyle{apalike}

\bibliography{references}
\end{document}